\documentclass[11pt, oneside]{amsart}

\usepackage{amsmath,amssymb,verbatim,amsthm,bbm, mathrsfs,amscd,pb-diagram}
\usepackage{aliascnt} %for environment names
\usepackage{appendix}
\DeclareMathAlphabet{\mathpzc}{OT1}{pzc}{m}{it}
\usepackage[all]{xy}

\usepackage{multirow,longtable}
\usepackage{diagbox}
\usepackage{makecell}
\usepackage{float} %used to force figures to stay in place (use \begin{figure}[H])
\usepackage{hyperref}
\usepackage{cleveref}
\usepackage{marginnote}
\usepackage{caption}

\usepackage{framed}

\usepackage{csquotes}

\usepackage{graphicx} % enable graphics, images and resizebox
\usepackage{lipsum}   % generate example text

\usepackage{tikz}

\usepackage{multicol}

\usepackage{lscape} % inorder to have landscape pages

\usepackage{framed}

\usepackage{cite}

\usepackage{array}
 \newcolumntype{H}{>{\setbox0=\hbox\bgroup}c<{\egroup}@{}}

\usepackage{todonotes}

\newcounter{todocounter}
\newcommand{\todonum}[1]{\stepcounter{todocounter}\todo{\thetodocounter: #1}}
\makeatletter
\providecommand\@dotsep{5}
\renewcommand{\listoftodos}[1][\@todonotes@todolistname]{%
	\@starttoc{tdo}{#1}}
\makeatother

\usepackage{float}
\restylefloat{table}

\crefname{table}{table}{tables}
\crefname{listing}{Program-code}{Program-codes}  
\Crefname{listing}{Program-code}{Program-codes}
\crefname{subsection}{subsection}{subsections}

\theoremstyle{plain}

\newtheorem{Thm}{Theorem}[section]
\newtheorem{Prop}[Thm]{Proposition}
\newtheorem{Lem}[Thm]{Lemma}

\theoremstyle{definition}
\newtheorem{Remark}[Thm]{Remark}

\numberwithin{equation}{section}
\newcommand{\Card}[1]{\left\vert #1\right\vert} %cardinality
\newcommand{\coset}[1]{\left[ #1 \right]}  %square brackets
\newcommand{\gen}[1]{\left\langle #1 \right\rangle}  %square brackets
\newcommand{\FNorm}[1]{\left\vert #1 \right\vert} %Norm in a local field/absolut value

\newcommand{\Image}{\operatorname{Im}}
\newcommand{\Ind}{i}

\newcommand{\Hom}{\operatorname{Hom}}

\newcommand{\C}{\mathbb C}

\newcommand{\R}{\mathbb{R}}
\newcommand{\N}{\mathbb{N}}

\newcommand{\bk}[1]{\left(#1\right)} %brackets
\newcommand{\bm}{\begin{multline*}}
\newcommand{\tu}{\end  {multline*}}

\DeclareMathOperator{\Id}{\mathbf{1}} %identity element 1
\renewcommand{\check}[1]{#1 ^{\vee}} %for coroots
\DeclareMathOperator{\Real}{Re} %Real part
\newcommand{\piece}[1]{\left\{\begin{matrix} #1 \end{matrix}\right.} %piecewise functions
\newcommand{\set}[1]{\left\{ #1 \right\}} %sets
\newcommand{\mvert}{\mathrel{}\middle\vert\mathrel{}} %midle vert line inside a set
\newcommand{\res}[1]{\Big\vert_{#1}}
\newcommand{\suml}{\sum\limits}
\newcommand{\prodl}{\, \prod\limits}

\newcommand{\rmod}{/}
\newcommand{\lmod}{\backslash}
\newcommand{\Stab}{\operatorname{Stab}}

\newcommand{\eeightchar}[8]{\renewcommand*{\arraystretch}{1} \begin{pmatrix}&& #2 &&&& \\ #1 & #3 & #4 & #5 & #6 &#7 & #8 \end{pmatrix} }

\newcommand{\dsixcharchar}[6]{\renewcommand*{\arraystretch}{1} \begin{bmatrix} &&& #6 & \\ #1 & #2 & #3 & #4 & #5 \end{bmatrix} }

\newcommand{\bfX}{\mathbf{X}}

\newcommand{\sepline}{\mbox{}\linebreak\noindent\rule{\textwidth}{2pt}\linebreak}

\newcommand{\redcolor}[1]{\mbox{}\\ \vspace{0.5cm}{\color{red} #1}\vspace{0.5cm}}
\newcommand{\bluecolor}[1]{\mbox{}\\ \vspace{0.5cm}{\color{blue} #1}\vspace{0.5cm}}

\makeatletter
\newcommand*{\rom}[1]{\expandafter\@slowromancap\romannumeral #1@}
\makeatother

\renewcommand*{\arraystretch}{1.1}

\makeatletter
\def\imod#1{\allowbreak\mkern10mu({\operator@font mod}\,\,#1)}
\makeatother

\makeatletter
\renewcommand\section{\@startsection{section}{1}{\z@}%
	{-3.5ex \@plus -1ex \@minus-.2ex}%
	{2.3ex \@plus.2ex}%
	{\center\normalfont\large\bfseries}}
\makeatother

\makeatletter
\renewcommand\subsection{\@startsection{subsection}{2}{\z@}%
	{-3.5ex \@plus -1ex \@minus-.2ex}%
	{2.3ex \@plus.2ex}%
	{\normalfont\large\bfseries}}
\makeatother

\makeatletter
\renewcommand\subsubsection{\@startsection{subsubsection}{3}{\z@}%
	{-3.5ex \@plus -1ex \@minus-.2ex}%
	{2.3ex \@plus.2ex}%
	{\normalfont\large\bfseries}}
\makeatother

\makeatletter
\newtheorem*{rep@theorem}{\rep@title} \newcommand{\newreptheorem}[2]{%
	\newenvironment{rep#1}[1]{%
		\def\rep@title{\bf #2 \ref{##1} }%
		\begin{rep@theorem} }%
		{\end{rep@theorem} } }
\makeatother
\newreptheorem{theorem}{Theorem}
\newreptheorem{lemma}{Lemma}
\protected\def\ignorethis#1\endignorethis{}
\let\endignorethis\relax

\usepackage{colortbl}
\usepackage{zref-savepos}
\usepackage{pict2e}

\newcounter{NoTableEntry}
\renewcommand*{\theNoTableEntry}{NTE-\the\value{NoTableEntry}}

\makeatletter
\newcommand*{\notableentry}{%
	\kern-\tabcolsep
	\stepcounter{NoTableEntry}%
	\vadjust pre{\zsavepos{\theNoTableEntry t}}% top
	\vadjust{\zsavepos{\theNoTableEntry b}}% bottom
	\zsavepos{\theNoTableEntry l}% left
	\raisebox{%
		\dimexpr\zposy{\theNoTableEntry b}sp
		-\zposy{\theNoTableEntry l}sp\relax
	}[0pt][0pt]{%
		\setlength{\unitlength}{1pt}%
		\edef\w{%
			\strip@pt\dimexpr\zposx{\theNoTableEntry r}sp%
			-\zposx{\theNoTableEntry l}sp\relax
		}% 
		\edef\h{%
			\strip@pt\dimexpr\zposy{\theNoTableEntry t}sp%
			-\zposy{\theNoTableEntry b}sp\relax
		}%
		\ifdim\w pt=0pt % prevent error in first run for \line(0,0)
		\else
		\begin{picture}(0,0)%
		\edef\x{%
			\noexpand\put(0,0){\noexpand\line(\w,\h){\w}}%  
			\noexpand\put(0,\h){\noexpand\line(\w,-\h){\w}}%
		}\x
		\end{picture}%
		\fi
	}%
	\hspace{0pt plus 1filll}%
	\zsavepos{\theNoTableEntry r}% right
	\kern-\tabcolsep
}

\makeatletter
\providecommand*{\cupdot}{%
	\mathbin{%
		\mathpalette\@cupdot{}%
	}%
}
\newcommand*{\@cupdot}[2]{%
	\ooalign{%
		$\m@th#1\cup$\cr
		\sbox0{$#1\cup$}%
		\dimen@=\ht0 %
		\sbox0{$\m@th#1\cdot$}%
		\advance\dimen@ by -\ht0 %
		\dimen@=.5\dimen@
		\hidewidth\raise\dimen@\box0\hidewidth
	}%
}

\providecommand*{\bigcupdot}{%
	\mathop{%
		\vphantom{\bigcup}%
		\mathpalette\@bigcupdot{}%
	}%
}
\newcommand*{\@bigcupdot}[2]{%
	\ooalign{%
		$\m@th#1\bigcup$\cr
		\sbox0{$#1\bigcup$}%
		\dimen@=\ht0 %
		\advance\dimen@ by -\dp0 %
		\sbox0{\scalebox{2}{$\m@th#1\cdot$}}%
		\advance\dimen@ by -\ht0 %
		\dimen@=.5\dimen@
		\hidewidth\raise\dimen@\box0\hidewidth
	}%
}
\makeatother
\allowdisplaybreaks

\newcommand{\fun}[1]{\bar{\omega}_{#1}}
\newcommand{\jac}[3]{r^{#1}_{#2}\bk{#3}}
\newcommand{\weyl}[1]{\mathit{W}_{#1}}
\newcommand{\w}{w}

\newcommand{\s}[1]{s_{#1}}
\newcommand{\para}[1]{#1}
\newcommand{\inner}[1]{\langle #1 \rangle}

\numberwithin{equation}{section}

\subjclass[2010]{22E50, 20G41, 20G05}
\usepackage{pifont}
\usepackage{slashbox}

\usepackage[showframe=false]{geometry}
\usepackage{changepage}

\usepackage{enumitem}
\usepackage{float}
\restylefloat{table}
\usepackage{xparse}% http://ctan.org/pkg/xparse
\NewDocumentCommand{\ceil}{s O{} m}{%
	\IfBooleanTF{#1} % starred
	{\left\lceil#3\right\rceil} % \ceil*[..]{..}
	{#2\lceil#3#2\rceil} % \ceil[..]{..}
}
\newcommand{\divides}{\Big \vert}
\newcommand{\RR}{\makecell{$red.^*$}}
\newcommand{\RI}{\notableentry}
\newcommand{\NRR}{\makecell{$red.$}}
\newcommand{\NRI}{\makecell{$irr.$}}

\title[Degenerate Principal Series of $E_8$]{The Degenerate Principal Series Representations of Exceptional Groups of Type $E_8$ over $p$-adic Fields} 
\author[Hezi Halawi and Avner Segal]{Hezi Halawi${^{1}}$ and Avner Segal${^{2}}$}
\address{${^1}$ School of Mathematics, Ben Gurion University of the Negev, POB 653, Be'er Sheva 84105, Israel}
\address{${^2}$ Mathematics Department, Shamoon College of Engineering,
	56 Bialik St., Beer-Sheva 84100, Israel}
\email{halawi@post.bgu.ac.il, avnerse@sce.ac.il}

\numberwithin{equation}{section}

\subjclass[2010]{22E50, 20G41, 20G05}

\begin{document}

%\begin{center}
%	\today
%\end{center}
\begin{abstract}
	In this paper, we study the reducibility of degenerate principal series of the simple, simply-connected exceptional group of type $E_8$.
	Furthermore, we calculate the maximal semi-simple subrepresentation and quotient of these representations for almost all cases.
\end{abstract}

\maketitle

\tableofcontents

%Missing cases:
%\[
%\coset{2,-1/2,1}, \coset{5,-1/2,1}.
%%, \coset{7,-3/2,1} .
%\]

\section{Introduction}
This paper is the final part in our project of studying the degenerate principal series of exceptional groups of type $E_n$.
This paper is about $E_8$, which, as often noted by David Kazhdan, is the smallest, split, simple, simply-connected, adjoint and simply laced group.
In fact, this essentially completes the study of degenerate principal series of simple $p$-adic groups up to isogeny.

More precisely, let $F$ be a non-Archimedean local field and let $G$ denote the split simple group of type $E_8$.
For a maximal parabolic subgroup $P$ of $G$ with a Levi subgroup $M$ and a $1$-dimensional representation $\Omega$ of $M$, we consider the following two questions:
\begin{itemize}
	\item Is the normalized parabolic induction $Ind_P^G\bk{\Omega}$ irreducible?
	\item If $Ind_P^G\bk{\Omega}$ is reducible, what is the length of its maximal semi-simple subrepresentation and quotient?
\end{itemize}
We completely answer the first one in \Cref{Thm:Main_theorem}, and almost completely answer the second.
In fact, there are only two pairs $\bk{P,\Omega}$ (out of hundreds of cases) in which we were only able to show that the maximal semi-simple subrepresentation is of length at most $2$.
In both of these cases, we show that the irreducible spherical subquotient is a subrepresentation and describe the other possible irreducible subrepresentation in terms of its Langlands data.
Further, we describe a decisive test to determine the length of the maximal semi-simple quotient for each of these cases, which would hopefully could be realized when stronger computing machines would be more commonly available.
These two cases are detailed in \Cref{Subsec:Unresolved_Cases}.

In order to answer the above questions, we use the algorithm described in \cite[Section 3]{SDPS_E6} and \cite[Section 3]{SDPS_E7}.
This provides an answer to both questions for almost all pairs $\bk{P, \Omega}$
For the remaining cases, not determined by the algorithm, further study is performed in \Cref{Sec:Exceptional_Cases}.
This project uses a script implemented in the Sagemath environment \cite{sagemath}.

%There are two exceptional cases where we can prove that the maximal semi-simple subrepresentation is of either length $1$ or $2$.
%We further show that their irreducible spherical subquotient is a subrepresentation and describe the other subquotient which is a candidate to being a subrepresentation in terms of its Langlands data.
%Finally, we describe a decisive test to determine the length of the maximal semi-simple quotient for each of this cases which would hopefully could be realized when stronger computing machines would be more commonly available.
%These two cases are detailed in \Cref{Subsec:Unresolved_Cases}.

The study of local degenerate principles series is useful for various reasons.
One of which is the study of the degenerate residual spectrum of the adelic group.
The study of the degenerate residual spectrum of the simple group of type $F_4$ is the topic of the first author's PhD dissertation and relies on the study of the local degenerate principal series of $F_4$, performed in \cite{MR2778237}.
The study of degenerate residual representations of $Spin_8$, with $P$ being the Heisenberg parabolic subgroup, was performed by the second author in \cite{MR3803152,MR4024536} and used to study exceptional $\theta$-lifts in \cite{RallisSchiffmannPaper}.

The study of the degenerate residual spectrum of the adelic groups of type $E_n$ is a work in progress as a joint project, \cite{DegResSpec_En}, of both authors.

This paper is structured as follows:
\begin{itemize}
	\item In \Cref{Sec:Preliminaries}, we recall basic notations and properties from representation theory of $p$-adic groups, we recall the algorithm described in \cite{SDPS_E6} and \cite{SDPS_E7} and we recall basic data on the exceptional group of type $E_8$.
	
	\item In \Cref{Sec:Main_Theorem} we state our main theorem, \Cref{Thm:Main_theorem}, and list all cases which can be resolved using our algorithm.
	This algorithm constitutes of reducibility and irreducibility test as well as some test to check if the representation admits a unique irreducible subrepresentation.
	
	\item In \Cref{Sec:Exceptional_Cases} we go over the exceptional cases which could not have been fully resolved by our algorithm.
	We resolve most of these cases completely and make some progress towards the resolution of the remaining two cases.
\end{itemize}

Finally, we wish to address a question which was broached to us following \cite{SDPS_E6} and \cite{SDPS_E7}.
The reducibility of a non-unitary degenerate principal series representation can be determined by the local Shahidi coefficients which would seem to make our algorithm obsolete.
However, the algorithm, presented in \Cref{Subsec:The_Algorithm}, is useful for various other reasons such as:
\begin{itemize}
	\item It allows us to determine the reducibility for unitary cases too.
	\item While local Shahidi coefficients can inform us regarding the reducibility of non-unitary degenerate principal series, the output of the reducibility test is also useful when studying global phenomenons such as the Siegel-Weil identity (and indeed this data is used in \cite{DegResSpec_En}).
	\item Data from the reducibility and irreducibility tests is useful for studying the structure of reducible degenerate principal series and in particular for studying its socle and cosocle.
\end{itemize}

We also wish to point out that comparing the lists of non-unitary reducible degenerate principal series determined by our method with that given by Shahidi's method, was useful for debugging purposes.
Indeed, as in the $E_6$ and $E_7$ case, our algorithm was decisive for all non-unitary cases.
The only cases where the algorithm was unable to determine the reducibility of the degenerate principal series were unitary ones.
%All exceptional cases where the reducibility could not be determined by the algorithm were for unitary representations.

%We point out that, as in our previous study of the degenerate principal series of groups of type $E_6$ and $E_7$, the algorithm determined completely the reducibility/irreducibility of all $i_{M_i,s,\chi}$ with $s\neq 0$ and the results match those determined by the local Shahidi coefficients.

\paragraph*{\textbf{Acknowledgments}}

%Parts of this work consists of parts of the first authors M.Sc. thesis \cite{HeziMScThesis}.
The second author was partially supported by grants 421/17 and 259/14 from the Israel Science Foundation as well as by the Junior Researcher Grant of Shamoon College of Engineering (SCE).

\section{Preliminaries}
\label{Sec:Preliminaries}

This section has three parts.
In the first part, we fix notations for this paper.
This part is organized as an enumerated list in order to make the look up of notations easier.

In the second part we recall the algorithm which we implemented for the study of the degenerate principal series of $E_8$.
%This algorithm can determine questions of reducibility and irreducibility or the length of maximal semi-simple subrepresentations or quotients for most cases.
%Those cases which could not be determined by the algorithm are dealt with in \Cref{Sec:Exceptional_Cases}.

Since the algorithm has been described in detail in \cite[Section 3]{SDPS_E6} and \cite[Section 3]{SDPS_E7}, we skip much of the details and describe only the broad strokes of it, while keeping all details which are required for \Cref{Sec:Main_Theorem} and \Cref{Sec:Exceptional_Cases}.

In the third part of this section we introduce the split group of type $E_8$.
%\todonum{Fill this paragraph}
%we set notations for this paper and recall the essentials of the algorithm used for the proof of \Cref{Thm:Main_theorem}.

\subsection{Groups, Characters and Representations}
\label{Subsec:Notations}
In this subsection, we fix notations and recall basic facts about the groups, characters and representations involved in this paper.
For a more detailed discussion, the reader is encouraged to consider \cite[Section 2]{SDPS_E6} and \cite[Section 2]{SDPS_E7}.

\subsubsection{Groups}
\begin{enumerate}
	\item Let $F$ be a non-Archimedean local field with norm $\FNorm{\cdot}$.
	\item Let $G$ denote the $F$-points of a split simply-connected reductive group.
%	 defined over $F$.
	\item Let $T$ be a maximal split torus of $G$.
	\item Let $B$ be a Borel subgroup of $G$ such that $T\subset B$.
	\item Let $\Phi_G$ denote the roots of $G$ with respect to $T$ and let $\Phi_G^{+}\subset \Phi_G$ denote the positive roots of $G$ with respect to $B$.
	\item Let $\Delta_G=\set{\alpha_1,...,\alpha_n}$ be the set of simple roots of $\Phi_G$ with respect to $B$.
	\item Let $n=\Card{\Delta_G}=\dim_F\bk{T}$ denote the rank of $G$.
	\item Let $\check{\Phi}_G = \set{\check{\alpha}\mvert \alpha\in\Phi_G}$ denote the set of coroots of $G$ with respect to $T$.
	
	\item We use $\inner{\cdot,\cdot}$ to denote the usual pairing between characters and co-characters of $T$.
	
	\item Let $\omega_{\alpha_1},...,\omega_{\alpha_n}$ denote the fundamental weights of $T$ which satisfy
	\[
	\gen{\omega_{\alpha_i},\check{\alpha_j}} = \delta_{i,j} .
	\]
	\item Let $W=\gen{s_i \mvert 1\leq i\leq n}$ denote the Weyl group of $G$ with respect to $T$, generated by the simple reflections $s_i$ associated with the simple roots $\alpha_i$.
	\item For $\Theta\subset\Delta_G$, let $P_\Theta=\gen{B,s_i\mvert \alpha_i\in\Theta}=M_\Theta\cdot U_\Theta$ be the standard parabolic subgroup of $G$ associated to $\Theta$.
	We denote its Levi subgroup by $M_\Theta$ and its unipotent radical by $U_\Theta$.
	\item Let $\Phi_{M_\Theta}$, $\Phi_{M_\Theta}^{+}$ and $\Delta_{M_\Theta}$ denote the roots, positive roots and simple roots of $M_\Theta$ with respect to $T$ and $B\cap M_\Theta$ respectively.
	\item Let $W_{M_\Theta}=\gen{s_i\mvert \alpha_i\in\Theta}$ denote the Weyl group of $M_\Theta$ with respect to $T$.
	\item Let $P_i = P_{\Delta_G\setminus\set{\alpha_i}}$ and $M_i=M_{\Delta_G\setminus\set{\alpha_i}}$ denote a maximal (proper) standard parabolic subgroup of $G$ and its Levi subgroup.
	
	\item For maximal standard Levi subgroups $M_i$ and $M_j$ of $G$, we write $M_{i,j}=M_i\cap M_j$.
	
	\item We denote the rank $1$ Levi subgroups by $L_i=M_{\set{\alpha_i}}$.
\end{enumerate}

\subsubsection{Characters}

\begin{enumerate}
	\item Let $\bfX\bk{G}=\set{\Omega:G\to\C^\times}$ denote the complex manifold of continuous characters of $G$, we use additive notations for this group, that is
	\[
	\bk{\Omega_1+\Omega_2}\bk{g} = \Omega_1\bk{g}\cdot\Omega_2\bk{g} .
	\]
	We usually use the letter $\Omega$ to denote elements in $\bfX\bk{M}$, for a non-minimal Levi subgroup $M$ of $G$, while using $\lambda$ to denote an element of $\bfX\bk{T}$.
	Also, note that $W$ acts on $\bfX\bk{T}$ via its action on $T$.
%	Also, note that $\bfX\bk{T}$ acts on $T$.
%	\todonum{Need also $\bfX^{un}\bk{T}$}
	\item We denote the set of unramified elements $\Omega\in \bfX\bk{T}$ by $\bfX^{un}\bk{T}$.
	\item Let $\Id_G\in\bfX\bk{G}$ denote the trivial character of $G$.
	\item We say that $\chi\in\bfX\bk{T}$ has finite order if there exists $k\in\N$ such that $\chi^k=\Id_T$.
	The order of $\chi$, denoted by $ord\bk{\chi}$, is the minimal $k\in\N$ such that $\chi^k=\Id_T$.
	
	In particular, every element $\Omega\in\bfX\bk{F^\times}$ can be written as $\Omega=s+\chi$, where $s\in\C$ and $\chi$ is of finite order.
	Namely,
	\[
	\Omega\bk{x} = \chi\bk{x}\FNorm{x}^s  \quad \forall x\in F^\times.
	\]
	It holds that $\bfX^{un}\bk{F^\times}$ can be described by all characters of the forms $\FNorm{x}^s$ for some $s\in\C$.
	
	\item We write $\Real\bk{\Omega}$ for the character
	\[
	\Real\bk{\Omega}\bk{x} = \FNorm{x}^{\Real\bk{s}}.
	\]

	\item We say that $\lambda\in\bfX\bk{T}$ is \textbf{anti-dominant} if
	\[
	\Real\bk{\inner{\lambda,\check{\alpha_i}}}\leq 0 \quad \forall 1\leq i\leq n.
	\]
	Note that every $W_G$-orbit in $\bfX\bk{T}$ contains at least one anti-dominant element  and all anti-dominant elements in the same $W_G$-orbit have an equal real part.
	As a convention, we denote an anti-dominant element by $\lambda_{a.d.}$.	

	\item Let $\Omega_{i,s,\chi}$ denote the character of a maximal Levi subgroup $M_i$ of $G$ associated with $\bk{s+\chi}\circ\omega_{\alpha_i}$, where $s\in\C$ and $\chi\in\bfX\bk{F^\times}$ is of finite order.
	Note that if $G$ is simple, then any element in $\bfX\bk{M_i}$ can be written this way.
%	\todonum{Normalized induction}

\end{enumerate}

\subsubsection{Representations}

\begin{enumerate}
	\item Let $Rep\bk{G}$ denote the category of admissible representations of $G$.
	\item As above, $\Id_G$ denotes the trivial representation of $G$.
	\item Let $i_M^G:Rep\bk{M}\to Rep\bk{G}$ and $r_M^G:Rep\bk{G}\to Rep\bk{M}$ denote the functors of normalized parabolic induction and Jacquet functor, adjunct by the Frobenius reciprocity:
	\begin{equation}
		\label{eq:Frobenius_reciprocity}
		\Hom_G\bk{\pi,i_M^G\sigma} \cong \Hom_M\bk{r_M^G\pi,\sigma} .
	\end{equation}
	
	\item For $\pi,\sigma\in Rep\bk{G}$, such that $\sigma$ is irreducible, let $mult\bk{\sigma,\pi}$ denote the multiplicity of $\sigma$ in the Jordan-H\"older series of $\pi$.
	
	\item Let $\mathfrak{R}\bk{G}$ denote the Grothendieck ring of $Rep\bk{G}$ and let $\coset{\pi}$ denote the image of $\pi\in Rep\bk{G}$ in $\mathfrak{R}\bk{G}$.
	Recall that $\mathfrak{R}\bk{G}$ admits a partial order such that $\pi_1\leq\pi_2$ if $mult\bk{\sigma,\pi_1}\leq mult\bk{\sigma,\pi_2}$ for every irreducible $\sigma\in Rep\bk{G}$.
	
	\item We remind the reader that, for Levi subgroups $L$ and $M$ of $G$ and $\sigma\in Rep\bk{M}$, the composition $\coset{r_L^G i_M^G\sigma}$ is given by the \emph{geometric lemma} (\cite[Lemma 2.12]{MR0579172}, \cite[Theorem 6.3.6]{Casselm1974}):
	\begin{equation}
	\label{Eq:gemoetric_lemma}
	\coset{r_L^G i_M^G\sigma} = \suml_{w\in W^{M,L}} \coset{i_{L'}^L \circ w \circ r_{M'}^M\sigma},
	\end{equation}
	where:
	\begin{itemize}
		\item $W^{M,L}=\set{w\in W\mvert w\bk{\Phi^{+}_M}\subseteq\Phi_G^{+},\ w^{-1}\bk{\Phi^{+}_L}\subseteq\Phi_G^{+}}$ is the set of shortest representatives in $W$ of the double coset space $W_L\lmod W_G\rmod W_M$.
		\item For $w\in W^{M,L}$ we write $M'=M\cap w^{-1}Lw$ and $L'=wMw^{-1}\cap L$.
	\end{itemize}

	\item For $\pi\in Rep\bk{G}$, we write $\coset{r_T^G\pi}=\suml_{i=1}^l n_i\times \coset{\lambda_i}$ for certain $\lambda_i\in\bfX\bk{T}$ such that $mult\bk{\lambda_i,\pi}=n_i>0$.
	Since $\dim_\C\bk{r_T^G\pi}$ is finite, there are only finitely many such $\lambda_i$.
	We call such $\lambda_i$ \textbf{the  exponents of $\pi$}.
	
	\item The representations $\pi=i_{M_i}^G\bk{\Omega_{M_i,s,\chi}}$ are called \textbf{degenerate principal series}.
	The exponent $\lambda_0=r_T^{M_i}\bk{\Omega_{M_i,s,\chi}}$ is called the \textbf{initial exponent} of $\pi$.
%	For short, we will also write $i_{M_i,s,\chi}$ for $i_{M_i}^G\bk{\Omega_{M_i,s,\chi}}$.
	
	\item We say that $\pi=i_{M_i}^G\bk{\Omega_{M_i,s,\chi}}$ is \textbf{regular} if $\Stab_W\bk{\lambda_0}=\set{1}$, where $\lambda_0=r_T^{M_i}\bk{\Omega_{M_i,s,\chi}}$.
	
	\item Let $w_0$ denote the longest element in $W^{M_i,T}$.
	It holds that $w_0\cdot\bk{i_{M_i}^G\bk{\Omega_{M_i,s,\chi}}}=i_{M_{j}}^G\bk{\Omega_{M_j,-s,\overline{\chi}}}$, where $\overline{\chi}$ is the complex conjugate of $\chi$ and $M_j = w_0 M_i w_0^{-1}$.
	We note that $M_j=M_i$, except when $G$ is of type $A_n$, $D_{2n+1}$ or $E_6$.
	We call $i_{M_{j}}^G\bk{\Omega_{M_j,-s,\overline{\chi}}}$ the \textbf{invert representation} of $i_{M_i}^G\bk{\Omega_{M_i,s,\chi}}$.
	We note that the invert representation has the same irreducible constituents but in an "inverted order".
	That is, $i_{M_{j}}^G\bk{\Omega_{M_j,-s,\overline{\chi}}}$ admits a Jordan-H\"older series whose irreducible quotients appear an inverted order than that of $i_{M_i}^G\bk{\Omega_{M_i,s,\chi}}$
	When $\chi=\Id$, the invert was defined in \cite[Remark 2.2.5]{MR1341660} as a variation of the Iwahori-Matsumoto involution.

\end{enumerate}

\subsection{The Algorithm}
\label{Subsec:The_Algorithm}

We now recall parts of the algorithm used by us to study the degenerate principal series $\pi=i_{M_i}^G\bk{\Omega_{M_i,s,\chi}}$.
For a more detailed account, the reader is encouraged to consider \cite[Section 3]{SDPS_E6} and \cite[Section 3]{SDPS_E7}.

We identify the data defining $\pi$ by a triple of numbers $\coset{i,s,ord\bk{\chi}}$.
In particular, the reducibility and lengths of the maximal semi-simple subrepresentation and quotient depend only on this triple (and are uniform among $\chi$s with the same order).

As explained in \cite[Remark 3.1]{SDPS_E6}, we may assume, without loss of generality, that $s\in\R$.
Furthermore, it is enough to consider only the cases where $s\leq 0$ since the invert representation of $\pi$ admits a Jordan-H\"older series with same irreducible quotients appearing but in inverted order.

The algorithm has the following parts:
\begin{enumerate}
	\item Determine all non-regular such $\pi$ - there is a finite number of such cases!
	\item Determine all reducible regular $\pi$ - there are only finitely many such cases!
	\item Apply reducibility tests to non-regular $\pi$ (these may be inconclusive).
	\item Apply an irreducibility test to non-regular $\pi$ (this too may be inconclusive). This test uses the so called \emph{branching rule calculation} introduced in \Cref{Subsubsec:Irreducibility}.
	
	\item Determine if $\pi$ admits a unique irreducible subrepresentation.
	As will be explained in \Cref{Subsubsec:UIS}, for $s<0$, $\pi$ admits a unique irreducible quotient.
%	 (this is enough due to contragredience and always true for $s>0$).
\end{enumerate}

The treatment of the regular cases is standard and is outlined in \cite[Subsection 3.1]{SDPS_E6}.
In what follows, we describe tools used  mostly for dealing with the non-regular case.
%Namely, bellow we give a short description of items (3)-(5) as it also affects the discussion of exceptional cases which are not answered by the algorithm (these are dealt with in \Cref{Sec:Exceptional_Cases}).

%We also wish to point out that, while the reducibility of degenerate principal series $\pi=i_{M_i}^G\bk{\Omega_{M_i,s,\chi}}$, with $s\neq 0$, can be determined using local Shahidi coefficient and without the use of this algorithm.
%However, this algorithm is still a useful tool for the following reasons:
%\begin{itemize}
%	\item It allows us to determine the reducibility in the case $s=0$ too.
%	\item While local Shahidi coefficients can inform us regarding the reducibility of $\pi$ when $s\neq0$, the output of the reducibility test is also useful when studying global phenomenons such as the Siegel-Weil identity.
%	\item Data from the reducibility and irreducibility tests is later used to determine the length of the maximal semi-simple subrepresentation.
%\end{itemize}
%
%We point out that, as in our previous study of the degenerate principal series of groups of type $E_6$ and $E_7$, the algorithm determined completely the reducibility/irreducibility of all $i_{M_i,s,\chi}$ with $s\neq 0$ and the results match those determined by the local Shahidi coefficients.
% See \cite[Proposition 5.3, 5.4]{MR1634020} for unramified characters.
%Also, we were able to determine the length of the maximal semi-simple subrepresentation and quotient.

%As mentioned above, some of the cases could not be determined by the algorithm.
%We deal with these cases in \Cref{Sec:Exceptional_Cases}.

\subsubsection{Reducibility Tests}
\label{Subsubsec:Reducibility}

In order to prove that a representation $\pi=i_{M_i}^G\bk{\Omega_{M_i,s,\chi}}$ is reducible, we provide another $\pi'\in Rep\bk{G}$ such that $\pi$ and $\pi'$ share a common irreducible constituent, while $\pi\neq \pi'$.
Usually, it is enough to consider $\pi'=i_{M_j}^G\bk{\Omega_{M_j,t,\chi^l}}$, where $l$ is a totative\footnote{That is, an integer $0<l\leq ord\bk{\chi}$ which is coprime with $ord\bk{\chi}$} of $ord\bk{\chi}$.
In other cases, one takes $\pi'=i_{M_{j_1,j_2}}^G\Omega_{s_1,s_2,\chi,k_1,k_2}$, where $\Omega_{s_1,s_2,\chi,k_1,k_2}\in\bfX\bk{M_{j_1,_2}}$ is associated with
\begin{equation}
	\label{Eq:1_dim_rep_of_almost_max_Levi}
	\bk{s_1+\chi^{k_1}}\circ\omega_{\alpha_{j_1}}+\bk{s_2+\chi^{k_2}}\circ\omega_{\alpha_{j_2}},
\end{equation}
such that at least one of $k_1$ and $k_2$ is a totative of $ord\bk{\chi}$.

In particular, one checks that:
\begin{itemize}
	\item $\pi$ and $\pi'$ share a common anti-dominant exponent and
	\item $r_T^G\pi\nleq r_T^G\pi'$.
\end{itemize}
Which guarantees that $\pi$ is reducible and shares a common irreducible subquotient with $\pi'$ (possibly $\pi'$ itself).
This is done via Tadi\'c's criterion (See \cite[Lemma 3.1]{MR1658535} and \cite[Section 3B]{SDPS_E7} for more information).

\subsubsection{Irreducibility Tests - Branching Rule Calculations}
\label{Subsubsec:Irreducibility}

We test the irreducibility of $\pi=i_{M_i}^G\bk{\Omega_{M_i,s,\chi}}$ using a method we call \emph{branching rule calculations}.
This method is also used, in some cases, to determine the number of irreducible subrepresentations of $\pi$.

Let
\[
\mathcal{S} = 
\set{f:\bfX\bk{T}\to\N \mvert \text{$f$ has a finite support}} .
\]
and note that $\mathcal{S}$ is endowed with a natural partial order.
For any $\sigma\in Rep\bk{G}$, let $f_\sigma\in\mathcal{S}$ be defined by
\[
f_\sigma\bk{\lambda} = mult\bk{\lambda,r_T^G\sigma} .
\]
In particular, for $\sigma=\pi$, $f_\pi$ can be calculated directly using \Cref{Eq:gemoetric_lemma}.
%Note that $mult\bk{\lambda,r_T^G\sigma}$ can be calculated using \Cref{Eq:gemoetric_lemma}

A sequence of functions $\set{f_j}_{j=0}^k$ is called a \textbf{unital $\sigma$-dominated sequence} if it satisfies the condition $f_0\leq f_1\leq ...\leq f_k\leq f_\sigma$ and $f_0=\delta_{\lambda}$ for some $\lambda\leq r_T^G\sigma$.
Here $\delta_{\lambda}$ denotes the Kronecker delta function given by
\[
\delta_\lambda\bk{\lambda'} = \piece{1,& \lambda'=\lambda \\ 0, & \lambda'\neq\lambda} .
\]

For a given $\lambda\leq r_T^G\sigma$, we describe a recipe of constructing a $\sigma$-dominated sequence $\set{f_j}_{j=0}^k$ such that $f_0=\delta_{\lambda}$.
%\[
%f_0\bk{\lambda'} = \piece{1,& \lambda'=\lambda \\ 0, & \lambda'\neq\lambda} .
%\]

The construction of the sequence $\set{f_j}_{j=0}^k$ is done using the following recursive process:
\begin{enumerate}
	\item Set $f_0=\delta_\lambda$.
	
	\item Assume that $f_0<...< f_l$ were determined. 
	We choose $\lambda'\in \operatorname{supp}\bk{f_l}$ and a Levi subgroup $L$ of $G$ such that $M$ admits a unique irreducible representation $\tau$ such that $\lambda'\leq r_T^L\tau$.
	
	\item For a choice of $\bk{\lambda',L,\tau}$ as above, for any $\mu\in \bfX\bk{T}$, let
	\[
	g\bk{\mu} = \max\set{f_l\bk{\mu}, \left\lceil \frac{f_l\bk{\lambda'}}{mult\bk{\lambda',r_T^L\tau}} \right\rceil \cdot mult\bk{\mu,r_T^L\tau}} .
	\]	
	
	\item If $f_l<g_{\lambda',L,\tau}$ for some choice of $\bk{\lambda',L,\tau}$, set $f_{l+1}=g$ and go back to step (2).
	Otherwise, we take $k=l$ and the process terminates.
\end{enumerate}

This process is later referred to as a \textbf{branching rule calculations}.

\begin{Remark}
	The uniqueness condition on $\tau$ in item (2) can be slightly relaxed.
	In fact, one actually only needs that $\coset{r_T^L \tau}$ be unique.
	Also, we note that one can apply this process to any function $f_0\leq f_\sigma$ and not only delta functions, in this case we call the resulting sequence a \textbf{non-unital $\sigma$-dominated sequence}.
\end{Remark}

A more detailed account of this recipe can be found in  \cite[Subsection 3.3]{SDPS_E6} and \cite[Subsection 3.3]{SDPS_E7}.
Furthermore, a pair of explicit examples of such calculations, in the case of $SL_4\bk{F}$, can be found in \cite[Appendix C]{SDPS_E7} and an explicit example of such calculations in the case of the group $E_6$ can be found in \cite[Appendix B]{SDPS_E6}.

The calculation of the sequence $f_0\leq f_1\leq...\leq f_k\leq f_\sigma$ as above relies on a database of irreducible representations of standard Levi subgroups of $G$,
the database used by us can be found in \Cref{App:Database}.

%In \cite[Subsection 3.3]{SDPS_E6} and \cite[Subsection 3.3]{SDPS_E7}, we describe a method to produce a certain sequence of functions $f_0\leq f_1\leq ...\leq f_k\leq f_\sigma$ such that
%\[
%f_0\bk{\lambda'} = \piece{1,& \lambda'=\lambda \\ 0, & \lambda'\neq\lambda} ,
%\]
%where $\lambda\leq r_T^G\sigma$ is fixed.
%Such a sequence $\set{f_j}_{j=0}^k$ is called \emph{$\sigma$-dominated}.
%Furthermore, in \cite[Appendix C]{SDPS_E7} the reader can find explicit examples of such calculations in the case of $SL_4\bk{F}$.

The terminal function $f_k$ in this sequence provides a lower bound to the multiplicities of exponents appearing in $r_T^G\sigma$.
Throughout this manuscript, we refer to this process as a \emph{branching rule calculation} (associated to $\lambda$).

In particular, for $\pi=i_{M_i}^G\bk{\Omega_{M_i,s,\chi}}$, let $\lambda_{a.d.}$ denote an anti-dominant exponent of $\pi$ and let  $\pi_0$ denote an irreducible subquotient of $\pi$ such that $\lambda_{a.d.}\leq r_T^G\pi_0$.
We construct a sequence $f_0\leq f_1\leq...\leq f_k\leq f_{\pi_0}$ as above.
If it holds that $f=f_\pi$, then it follows that $r_T^G\pi=r_T^G\pi_0$ and since all subquotients of $\pi$ have their cuspidal support along $B$, it follows that $\pi=\pi_0$.
In particular, $\pi$ is irreducible.

In fact, in order to prove that $\pi$ is irreducible, it is enough to show that
\[
\begin{array}{l}
mult\bk{\lambda_0,r_T^G\pi_0} = mult\bk{\lambda_0,r_T^G\pi} \\
mult\bk{\lambda_1,r_T^G\pi_0} = mult\bk{\lambda_1,r_T^G\pi} ,
\end{array}
\]
where $\lambda_0$ is the initial exponent of $\pi$ and $\lambda_1$ is its terminal exponent, namely the initial exponent of the invert of $\pi$.
This would imply that $\pi$ is both its unique irreducible subrepresentation and quotient and hence it is irreducible.

\subsubsection{Irreducible Subrepresentations}
\label{Subsubsec:UIS}

Bellow, we describe a few methods that can be realized as part of the algorithm, in order to determine the length of the maximal semi-simple subrepresentation of the degenerate principal series representation $\pi=i_{M_i,s,\chi}$.

%Bellow, we describe a few methods that were realized by us, as part of the algorithm, to determine that degenerate principal series representation $\pi=i_{M_i,s,\chi}$ admits a unique irreducible subrepresentation.
%We note that this property could not be determined by these methods for a number of cases, whether the consequent is true or false.
%That is, the algorithm would be inconclusive when the length of the semi-simple subrepresentation is not irreducible and also in certain cases where it is.
%These cases are listed in \Cref{Thm:Main_theorem} and are dealt with in \Cref{Sec:Exceptional_Cases}.

%antecedent is false while the consequent is true or because the consequent

%Here we describe the methods realized by us as part of the algorithm do determine when a degenerate principal series representation $\pi=i_{M_i,s,\chi}$ admits a unique irreducible quotient.
%For exceptional cases which were not solved by the algorithm or have a maximal semi-simple subrepresentation of length $2$ (there are no other options in the case of $E_8$) see \Cref{Sec:Exceptional_Cases}.
%\todonum{Rewrite this paragraph.}

\begin{enumerate}
	\item If $mult\bk{\lambda_0,r_T^G\pi}=1$, where $\lambda_0$ is the initial exponent of $\pi$, then $\pi$ admits a unique irreducible subrepresentation.
	
	\item Let $\lambda_{a.d.}$ denote an anti-dominant exponent of $\pi$ and let $\pi_0$ denote an irreducible subquotient of $\pi$ such that $\lambda_{a.d.}\leq r_T^G\pi_0$.
	If
	\begin{equation}
	\label{Eq:Mult_condition_UIS}
	mult\bk{\lambda_0,r_T^G\pi_0}=mult\bk{\lambda_0,r_T^G\pi},
	\end{equation}	
%	\begin{equation}
%	\label{Eq:Mult_condition_UIS}
%	mult\bk{\lambda_{a.d.},r_T^G\pi_0}=mult\bk{\lambda_{a.d.},r_T^G\pi},
%	\end{equation}
	then $\pi_0$ is the unique irreducible subrepresentation of $\pi$.
	
	The condition in \Cref{Eq:Mult_condition_UIS} can, in many cases, be verified using branching rule calculations.
	Also, this argument seems to be relevant only when $s\leq 0$, since otherwise $\pi$ would be irreducible (and thus would not be the subject of investigation for this part of the algorithm).
	
	\item Given a unitary degenerate principal series $\pi=i_{M_i}^G\bk{\Omega_{M_i,0,\chi}}$, then $\pi$ is semi-simple of length at most $2$ by \cite[Lemma 5.2]{MR1951440}.
	In particular, if it is reducible, it is of length $2$.

\end{enumerate}

We note that, in a number of cases, the length of the semi-simple subrepresentation cannot be determined by these methods.
This can happen for cases where the length of the semi-simple subrepresentation can be either $1$ or more.
These cases are listed in \Cref{Thm:Main_theorem} and are dealt with in \Cref{Sec:Exceptional_Cases} using various other methods.

\begin{Remark}
	For $s>0$, the multiplicity of the initial exponent always appears in $r_T^G\pi$ with multiplicity $1$ and thus $\pi$ admits a unique irreducible subrepresentation.
	Conversely, this is why $\pi$ always admits a unique irreducible quotient when $s<0$.
\end{Remark}

\subsection{The Exceptional Group of Type $E_8$}
\label{Subsec:GroupData}

Let $G$ be the split, semi-simple, simply-connected group of type $E_8$.
In this subsection we describe the structure of $G$.
We fix a Borel subgroup $\para{B}$ and a maximal split torus $T\subset \para{B}$ with notations as in \Cref{Subsec:Notations}.
The set of roots, $\Phi_{G}$, contains $240$ roots. The group $G$ is generated by symbols \[\set{x_{\alpha}(r) \: : \: \alpha \in \Phi_{G} ,r \in F}\]
subject to the Chevalley relations as in \cite[Section 6]{MR0466335}.

We label the simple roots $\Delta_{G}$ and the Dynkin diagram of $G$ using the Bourbaki labelling:
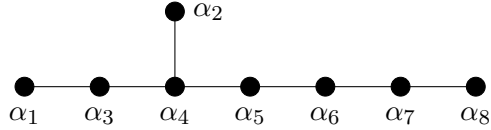
\[\begin{tikzpicture}[scale=0.5]
\draw (-1,0) node[anchor=east]{};
\draw (0 cm,0) -- (12 cm,0);
\draw (4 cm, 0 cm) -- +(0,2 cm);
\draw[fill=black] (0 cm, 0 cm) circle (.25cm) node[below=4pt]{$\alpha_1$};
\draw[fill=black] (2 cm, 0 cm) circle (.25cm) node[below=4pt]{$\alpha_3$};
\draw[fill=black] (4 cm, 0 cm) circle (.25cm) node[below=4pt]{$\alpha_4$};
\draw[fill=black] (6 cm, 0 cm) circle (.25cm) node[below=4pt]{$\alpha_5$};
\draw[fill=black] (8 cm, 0 cm) circle (.25cm) node[below=4pt]{$\alpha_6$};
\draw[fill=black] (10 cm, 0 cm) circle (.25cm) node[below=4pt]{$\alpha_7$};
\draw[fill=black] (12 cm, 0 cm) circle (.25cm) node[below=4pt]{$\alpha_8$};
\draw[fill=black] (4 cm, 2 cm) circle (.25cm) node[right=3pt]{$\alpha_2$};
\end{tikzpicture}\]

Recall that for  $\Theta \subset \Delta_{G}$ we denote by $M_{\Theta}$ the standard Levi subgroup  of $G$ such that $\Delta_{M} = \Theta$. We let $M_{i}$ denote the Levi subgroup of the maximal parabolic subgroup $\para{P}_{i}= \para{P}_{\Delta_{G} \setminus \set{\alpha_i}}$. 

\begin{Lem}
	Under these notations, it holds that:
	\begin{enumerate}[ref =\Cref{Lemma::structrue::E7}.(\arabic*)] 
		\item $M_1\cong \set{g\in GSpin_{14}\bk{F} \mvert \det\bk{g}\in \bk{F^\times}^2}$.
%		, where $\det$ is the similitude factor of $GSpin_{14}$.
%		RM1=(4, 5, 7, 10, 8, 6, 4, 2)
%		RM1\cap M1^der = {t^4=1}
%		M1^der=Spin_{14} --> 2
		\item $M_2\cong GL_8\bk{F}$.
%		\item $M_2\cong \set{g\in GL_8\bk{F} \mvert \det\bk{g}\in \bk{F^\times}^2}$.
%		RM2=(5, 8, 10, 15, 12, 9, 6, 3)
%		RM2\cap M2^der = {t^8=1}
%		M2^der=SL_8 --> 8
		\item $M_3\cong \set{\bk{g_1,g_2}\in GL_2\bk{F}\times GL_7\bk{F} \mvert \det \bk{g_1} = \det \bk{g_2}}$.
%		RM3=(7, 10, 14, 20, 16, 12, 8, 4)
%		RM3\cap M3^der = {t^14=1}
%		M3^der=SL_2\times SL_7 --> 2*7=14
		\item $M_4\cong \set{\bk{g_1,g_2,g_3}\in GL_3\bk{F}\times GL_2\bk{F}\times GL_5\bk{F} \mvert \det \bk{g_1} = \det \bk{g_2} = \det\bk{g_3}}$.
%		RM4=(10, 15, 20, 30, 24, 18, 12, 6)
%		RM4\cap M4^der = {t^30=1}
%		M3^der = SL_3\times SL_2 \times SL_5 --->3*2*5=30
		\item $M_5\cong \set{\bk{g_1,g_2}\in GL_5\bk{F}\times GL_4\bk{F} \mvert \det \bk{g_1} = \det \bk{g_2}}$.
%		RM5=(8, 12, 16, 24, 20, 15, 10, 5)
%		RM5\cap M5^der = {t^20=1}
%		M5^der = SL_5\times SL_4 ---> 5*4=20
		\item $M_6\cong \set{\bk{g_1,g_2}\in GSpin_{10}\bk{F}\times GL_3\bk{F} \mvert \det \bk{g_1} = \det \bk{g_2} \in \bk{F^\times}^2}$.
%		RM6=(6, 9, 12, 18, 15, 12, 8, 4)
%		RM6\cap M6^der = {t^12=1}
%		M6^der = Spin_10\times SL_3 ---> 2*3=6
		\item $M_7\cong \set{\bk{g_1,g_2}\in GE_6\bk{F}\times GL_2\bk{F} \mvert \det \bk{g_1} = \det \bk{g_2}}$.
%		, where the $\det$ on the LHS is the similitude factor of $GE_6$.
%		RM7=(4, 6, 8, 12, 10, 8, 6, 3)
%		RM7\cap M7^der = {t^6=1}
%		M7^der = E_6\times SL_2 ---> 3*2=6
		\item $M_8\cong GE_7\bk{F}$.
%		RM8=(2, 3, 4, 6, 5, 4, 3, 2)
%		RM7\cap M7^der = {t^2=1}
%		M7^der = E_7 ---> 2
	\end{enumerate} 
	Where $\det$ denotes the similitude factors on the relevant groups (in particular, the similitude factor on $GL_n$ is the usual determinant).
\end{Lem}
We record here, for $1 \leq i \leq 8$, the cardinality of $W^{M_i,T}$, the set of shortest representatives of $\weyl{G} \slash \weyl{M_i}$
\begin{center}
	\begin{tabular}{|c|c|c|c|c|c|c|c|c|}
		\hline
		$i$ & 1 & 2 & 3 & 4 & 5 & 6 & 7 & 8 \\
		\hline 
		$|W^{M_i,T}|$ & 2,160 & 17,280 & 69,120 & 483,840 & 241,920 & 60,480 & 6,720 & 240 \\
		\hline
	\end{tabular} 
\end{center}
\vspace{0.5cm}
We also mention that $|\weyl{G}| =  696,729,600$.
Every $\lambda \in \bfX(T)$ is of the form 
\[\lambda = \sum_{i=1}^{8} \Omega_{i} \circ \fun{\alpha_i}.\]
As a shorthand, we will write
\[\eeightchar{\Omega_1}{\Omega_2}{\Omega_3}{\Omega_4}{\Omega_5}{\Omega_6}{\Omega_7}{\Omega_8}=\sum_{i=1}^{8} \Omega_{i} \circ \fun{\alpha_i}.\]

%Also, let 
%$\eeightcharp{\Omega_1}{\Omega_2}{\Omega_3}{\Omega_4}{\Omega_5}{\Omega_6}{\Omega_7}{\Omega_8}$ denote its class in $\mathfrak{R}(T)$.

\section{The Main Theorem}
\label{Sec:Main_Theorem}

\begin{Thm}
	\label{Thm:Main_theorem}
	Let $\pi=\Ind_{\para{M}_i}^{G}(\Omega_{\para{M}_i,s,\chi})$ with $s\leq 0$ and let $k=ord(\chi)$.
	\begin{enumerate}
	\item
	The following tables
	Tables \ref{Tab::E7::REG::P_1}-\ref{Tab::E7::REG::P_8} bellow lists all triples $[i,s,k]$ such that $\pi$ is either non-regular or reducible.
	In particular, for each triple $[i,s,k]$ the entry in the $i$th table for this value of $s$ and $k$ will be
	\begin{itemize}
		\item \textbf{irr.} for non-regular and irreducible $\pi$. %%% \NRI
		\item \textbf{red.} for non-regular and reducible $\pi$. %%% \NRR
		\item \textbf{red.${}^*$} for regular and reducible $\pi$. %%% \RR
	\end{itemize}
	For any triple $[i,s,k]$, not appearing in the tables, the degenerate principal series $\Ind_{M_i}^{G}(\Omega_{M_i,s,\chi})$, with $ord(\chi)=k$, is regular and irreducible.

	\item 
	All	$\pi = \Ind_{M_{i}}^{G}\bk{\Omega_{M_i,s,\chi}}$	admit a unique irreducible subrepresentation, with the exception of:
	\begin{enumerate}
		\item $[i,s,k]$ is one of $[1,-5/2,1]$, $[3,-1/2,2]$ $[6,0,1]$, $[6,0,2]$ and $[7,-3/2,1]$. 
		In these cases, the representation $\pi$ admits a maximal semi-simple subrepresentation of length $2$.
		\item $[i,s,k]$ is one of $[2,-1/2,1]$ and $[5,-1/2,1]$, in which case the length of the maximal semi-simple subrepresentation of $\pi$ is at most $2$.
	\end{enumerate}

\begin{comment}
		\item
	%	With the exception of the  cases listed below,
	%	$\pi = \Ind_{M_{i}}^{G}\bk{\Omega_{M_i,s,\chi}}$
	%	admits a unique irreducible subrepresentation. In the remaining cases, listed below, the length of the maximal semi-simple subrepresentation is $2$: 
	All	$\pi = \Ind_{M_{i}}^{G}\bk{\Omega_{M_i,s,\chi}}$	admit a unique irreducible subrepresentation, with the exception of the following representations, which admit a maximal semi-simple subrepresentation of length $2$:
	\begin{enumerate}
		%		$[6,0,1]$, $[6,0,2]$, $[3,-1/2,2]$, $[1,-5/2,1]$
		\item
		$\Ind_{M_{1}}^{G}(\Omega_{M_{1},-\frac{5}{2},triv})$.
		\item
		$\Ind_{M_{3}}^{G}(\Omega_{M_{3},-\frac{1}{2},\chi})$,  where $ord(\chi)=2$.
		\item
		$\Ind_{M_{6}}^{G}(\Omega_{M_{6},0,\chi})$,  where $\chi^2=1$.
		\item
		$\Ind_{M_{7}}^{G}(\Omega_{M_{7},-\frac{3}{2},triv})$.
	\end{enumerate}
	And the exception of $\Ind_{M_{2}}^{G}(\Omega_{M_{2},-\frac{1}{2},triv})$ and $\Ind_{M_{5}}^{G}(\Omega_{M_{5},-\frac{1}{2},triv})$, where the length of the maximal semi-simple subrepresentation is at most $2$.
	%	\todonum{Possibly also [2,-1/2,1] and [5,-1/2,1].}
\end{comment}
\end{enumerate}
\end{Thm}

\begin{landscape}
	\begin{enumerate}
	\item 
	For $\para{P}=\para{P}_{1}$ 
	\begin{center} % \footnotesize  
		\begin{longtable}{|c|c|c|c|c|c|c|c|c|c|c|c|c|c|c|c|c|} 
			\hline 
			\diagbox{$ord\bk{\chi}$}{$s$}  & $-\frac{23}{2}$  & $-\frac{21}{2}$  & $-\frac{19}{2}$  & $-\frac{17}{2}$  & $-\frac{15}{2}$  & $-\frac{13}{2}$  & $-\frac{11}{2}$  & $-\frac{9}{2}$  & $-\frac{7}{2}$  & $-3$  & $-\frac{5}{2}$  & $-2$  & $-\frac{3}{2}$  & $-1$  & $-\frac{1}{2}$  & $0$  
			\\ \hline 
			$ 1 $ & \RR & \NRI & \NRI & \NRR & \NRI & \NRR & \NRR & \NRI & \NRR & \NRI & \NRR & \NRI & \NRI & \NRI & \NRR & \NRI  
			\\ \hline 
			$ 2 $ & \RI & \RI & \RI & \RI & \RI & \RI & \RI & \RI & \RR & \NRI & \NRI & \NRI & \NRI & \NRI & \NRR & \NRI  
			\\ \hline 
			\caption{$\para{P}_{1}$-Reducibility Points}
			\label{Tab::E7::REG::P_1} 
		\end{longtable} 
		
	\end{center} 
	\item 
	For $\para{P}=\para{P}_{2}$ 
	\begin{center} % \footnotesize  
		\begin{longtable}{|c|c|c|c|c|c|c|c|c|c|c|c|c|c|c|c|c|c|c|c|c|c|c|c|c|c|c|c|c|c|c|c} 
			\hline 
			\diagbox{$ord\bk{\chi}$}{$s$}  & $-\frac{17}{2}$  & $-\frac{15}{2}$  & $-\frac{13}{2}$  & $-\frac{11}{2}$  & $-\frac{9}{2}$  & $-\frac{7}{2}$  & $-3$  & $-\frac{5}{2}$  & $-2$  & $-\frac{3}{2}$  & $-\frac{7}{6}$  & $-1$  & $-\frac{5}{6}$  & $-\frac{1}{2}$  & $-\frac{1}{6}$  & $0$  
			\\ \hline 
			$ 1 $ & \RR & \NRI & \NRR & \NRR & \NRR & \NRR & \NRI & \NRR & \NRI & \NRR & \NRI & \NRI & \NRI & \NRR & \NRI & \NRI  
			\\ \hline 
			$ 2 $ & \RI & \RI & \RI & \RI & \RI & \RR & \NRI & \NRR & \NRI & \NRR & \RI & \NRI & \RI & \NRR & \RI & \NRI  
			\\ \hline 
			$ 3 $ & \RI & \RI & \RI & \RI & \RI & \RI & \RI & \RI & \RI & \RR & \NRI & \RI & \NRI & \NRI & \NRI & \RI  
			\\ \hline 
			\caption{$\para{P}_{2}$-Reducibility Points}
			\label{Tab::E7::REG::P_2} 
		\end{longtable} 
		
	\end{center} 
	\item 
	For $\para{P}=\para{P}_{3}$ 
	\begin{center} % \footnotesize 
		\begin{longtable}{|c|c|c|c|c|c|c|c|c|c|c|c|c|c|c|c|c|c|c|c|c|c|c|c|c|c|c|c|c|c|c|c} 
			\hline 
			\diagbox{$ord\bk{\chi}$}{$s$}  & $-\frac{13}{2}$  & $-\frac{11}{2}$  & $-\frac{9}{2}$  & $-\frac{7}{2}$  & $-3$  & $-\frac{5}{2}$  & $-2$  & $-\frac{3}{2}$  & $-\frac{7}{6}$  & $-1$  & $-\frac{5}{6}$  & $-\frac{3}{4}$  & $-\frac{1}{2}$  & $-\frac{1}{4}$  & $-\frac{1}{6}$  & $0$  
			\\ \hline 
			$ 1 $ & \RR & \NRR & \NRR & \NRR & \NRI & \NRR & \NRR & \NRR & \NRR & \NRR & \NRI & \NRI & \NRR & \NRI & \NRI & \NRI  
			\\ \hline 
			$ 2 $ & \RI & \RI & \RI & \RR & \NRI & \NRR & \NRR & \NRR & \RI & \NRR & \RI & \NRI & \NRR & \NRI & \RI & \NRI  
			\\ \hline 
			$ 3 $ & \RI & \RI & \RI & \RI & \RI & \RI & \RI & \RR & \NRR & \RI & \NRI & \RI & \NRI & \RI & \NRI & \RI  
			\\ \hline 
			$ 4 $ & \RI & \RI & \RI & \RI & \RI & \RI & \RI & \RI & \RI & \RR & \RI & \NRI & \NRI & \NRI & \RI & \NRI  
			\\ \hline 
			\caption{$\para{P}_{3}$-Reducibility Points} 
			\label{Tab::E7::REG::P_3} 
		\end{longtable} 
		
	\end{center} 
	\item 
	For $\para{P}=\para{P}_{4}$ 
	\begin{center} % \footnotesize 
		\begin{longtable}{|c|c|c|c|c|c|c|c|c|c|c|c|c|c|c|c|c|c|c|c|c|c|c|c|c|c|c|c|c|c|c|c} 
			\hline 
			\diagbox{$ord\bk{\chi}$}{$s$}  & $-\frac{9}{2}$  & $-\frac{7}{2}$  & $-\frac{5}{2}$  & $-2$  & $-\frac{3}{2}$  & $-\frac{7}{6}$  & $-1$  & $-\frac{5}{6}$  & $-\frac{3}{4}$  & $-\frac{1}{2}$  & $-\frac{1}{3}$  & $-\frac{3}{10}$  & $-\frac{1}{4}$  & $-\frac{1}{6}$  & $-\frac{1}{10}$  & $0$  
			\\ \hline 
			$ 1 $ & \RR & \NRR & \NRR & \NRR & \NRR & \NRR & \NRR & \NRR & \NRR & \NRR & \NRI & \NRR & \NRI & \NRI & \NRI & \NRI  
			\\ \hline 
			$ 2 $ & \RI & \RI & \RR & \NRR & \NRR & \RI & \NRR & \RI & \NRR & \NRR & \NRI & \RI & \NRI & \NRI & \RI & \NRI  
			\\ \hline 
			$ 3 $ & \RI & \RI & \RI & \RI & \RR & \NRR & \RI & \NRR & \RI & \NRR & \NRI & \RI & \RI & \NRI & \RI & \NRI  
			\\ \hline 
			$ 4 $ & \RI & \RI & \RI & \RI & \RI & \RI & \RR & \RI & \NRR & \NRR & \RI & \RI & \NRI & \RI & \RI & \NRI  
			\\ \hline 
			$ 5 $ & \RI & \RI & \RI & \RI & \RI & \RI & \RI & \RI & \RI & \RR & \RI & \NRR & \RI & \RI & \NRI & \RI  
			\\ \hline 
			$ 6 $ & \RI & \RI & \RI & \RI & \RI & \RI & \RI & \RI & \RI & \RR & \NRI & \RI & \RI & \NRI & \RI & \NRI  
			\\ \hline 
			\caption{$\para{P}_{4}$-Reducibility Points} 
			\label{Tab::E7::REG::P_4} 
		\end{longtable} 
		
	\end{center} 
	\item 
	For $\para{P}=\para{P}_{5}$ 
	\begin{center} % \footnotesize 
		\begin{longtable}{|c|c|c|c|c|c|c|c|c|c|c|c|c|c|c|c|c|c|c|c|c|c|c|c|c|c|c|c|c|c|c|c} 
			\hline 
			\diagbox{$ord\bk{\chi}$}{$s$}  & $-\frac{11}{2}$  & $-\frac{9}{2}$  & $-\frac{7}{2}$  & $-\frac{5}{2}$  & $-2$  & $-\frac{3}{2}$  & $-\frac{7}{6}$  & $-1$  & $-\frac{5}{6}$  & $-\frac{3}{4}$  & $-\frac{1}{2}$  & $-\frac{3}{10}$  & $-\frac{1}{4}$  & $-\frac{1}{6}$  & $-\frac{1}{10}$  & $0$  
			\\ \hline 
			$ 1 $ & \RR & \NRR & \NRR & \NRR & \NRR & \NRR & \NRR & \NRR & \NRR & \NRI & \NRR & \NRI & \NRI & \NRI & \NRI & \NRI  
			\\ \hline 
			$ 2 $ & \RI & \RI & \RI & \RR & \NRR & \NRR & \RI & \NRR & \RI & \NRI & \NRR & \RI & \NRI & \RI & \RI & \NRI  
			\\ \hline 
			$ 3 $ & \RI & \RI & \RI & \RI & \RI & \RR & \NRR & \RI & \NRR & \RI & \NRR & \RI & \RI & \NRI & \RI & \RI  
			\\ \hline 
			$ 4 $ & \RI & \RI & \RI & \RI & \RI & \RI & \RI & \RR & \RI & \NRI & \NRR & \RI & \NRI & \RI & \RI & \NRI  
			\\ \hline 
			$ 5 $ & \RI & \RI & \RI & \RI & \RI & \RI & \RI & \RI & \RI & \RI & \RR & \NRI & \RI & \RI & \NRI & \RI  
			\\ \hline 
			\caption{$\para{P}_{5}$-Reducibility Points} 
			\label{Tab::E7::REG::P_5} 
		\end{longtable} 
		
	\end{center} 
	
	\item 
	For $\para{P}=\para{P}_{6}$ 
	\begin{center} % \footnotesize 
		\begin{longtable}{|c|c|c|c|c|c|c|c|c|c|c|c|c|c|c|c|c|c|c|c|c|c|c|c|c|c|c|c|c|c|c|c} 
			\hline 
			\diagbox{$ord\bk{\chi}$}{$s$}  & $-7$  & $-6$  & $-5$  & $-4$  & $-3$  & $-\frac{5}{2}$  & $-2$  & $-\frac{5}{3}$  & $-\frac{3}{2}$  & $-\frac{4}{3}$  & $-1$  & $-\frac{2}{3}$  & $-\frac{1}{2}$  & $-\frac{1}{3}$  & $-\frac{1}{4}$  & $0$  
			\\ \hline 
			$ 1 $ & \RR & \NRR & \NRR & \NRR & \NRR & \NRR & \NRR & \NRI & \NRI & \NRI & \NRR & \NRI & \NRR & \NRI & \NRI & \NRR  
			\\ \hline 
			$ 2 $ & \RI & \RI & \RI & \RI & \RR & \NRR & \NRR & \RI & \NRI & \RI & \NRR & \RI & \NRR & \RI & \NRI & \NRR  
			\\ \hline 
			$ 3 $ & \RI & \RI & \RI & \RI & \RI & \RI & \RR & \NRI & \RI & \NRI & \NRR & \NRI & \RI & \NRI & \RI & \NRI  
			\\ \hline 
			$ 4 $ & \RI & \RI & \RI & \RI & \RI & \RI & \RI & \RI & \RI & \RI & \RI & \RI & \RR & \RI & \NRI & \NRI  
			\\ \hline 
			\caption{$\para{P}_{6}$-Reducibility Points} 
			\label{Tab::E7::REG::P_6} 
		\end{longtable} 
		
	\end{center} 
	\item 
	For $\para{P}=\para{P}_{7}$ 
	\begin{center} % \footnotesize 
		\begin{longtable}{|c|c|c|c|c|c|c|c|c|c|c|c|c|c|c|c|c|c|c|c|c|c|c|c|c|c|c|c|c|c|c|c} 
			\hline 
			\diagbox{$ord\bk{\chi}$}{$s$}  & $-\frac{19}{2}$  & $-\frac{17}{2}$  & $-\frac{15}{2}$  & $-\frac{13}{2}$  & $-\frac{11}{2}$  & $-\frac{9}{2}$  & $-4$  & $-\frac{7}{2}$  & $-3$  & $-\frac{5}{2}$  & $-2$  & $-\frac{3}{2}$  & $-1$  & $-\frac{1}{2}$  & $-\frac{1}{6}$  & $0$  
			\\ \hline 
			$ 1 $ & \RR & \NRR & \NRI & \NRI & \NRR & \NRR & \NRI & \NRI & \NRI & \NRR & \NRI & \NRR & \NRI & \NRR & \NRI & \NRI  
			\\ \hline 
			$ 2 $ & \RI & \RI & \RI & \RI & \RI & \RR & \NRI & \NRI & \NRI & \NRR & \NRI & \NRI & \NRI & \NRR & \RI & \NRI  
			\\ \hline 
			$ 3 $ & \RI & \RI & \RI & \RI & \RI & \RI & \RI & \RI & \RI & \RI & \RI & \RI & \RI & \RR & \NRI & \RI  
			\\ \hline 
			\caption{$\para{P}_{7}$-Reducibility Points} 
			\label{Tab::E7::REG::P_7} 
		\end{longtable} 
		
	\end{center} 
	\item 
	For $\para{P}=\para{P}_{8}$ 
	\begin{center} % \footnotesize 
		\begin{longtable}{|c|c|c|c|c|c|c|c|c|c|c|c|c|c|c|c|c|c|c|c|c|c|c|c|c|c|c|c|c|c|c|c} 
			\hline 
			\diagbox{$ord\bk{\chi}$}{$s$}  & $-\frac{29}{2}$  & $-\frac{27}{2}$  & $-\frac{25}{2}$  & $-\frac{23}{2}$  & $-\frac{21}{2}$  & $-\frac{19}{2}$  & $-\frac{17}{2}$  & $-\frac{15}{2}$  & $-\frac{13}{2}$  & $-\frac{11}{2}$  & $-\frac{9}{2}$  & $-\frac{7}{2}$  & $-\frac{5}{2}$  & $-\frac{3}{2}$  & $-\frac{1}{2}$  & $0$  
			\\ \hline 
			$ 1 $ & \RR & \NRI & \NRI & \NRI & \NRI & \NRR & \NRI & \NRI & \NRI & \NRR & \NRI & \NRI & \NRI & \NRI & \NRR & \NRI  
			\\ \hline 
			$ 2 $ & \RI & \RI & \RI & \RI & \RI & \RI & \RI & \RI & \RI & \RI & \RI & \RI & \RI & \RI & \RR & \NRI  
			\\ \hline 
			\caption{$\para{P}_{8}$-Reducibility Points} 
			\label{Tab::E7::REG::P_8}
		\end{longtable} 
		
	\end{center} 
\end{enumerate}
\end{landscape} 

%	\item
%%	With the exception of the  cases listed below,
%%	$\pi = \Ind_{M_{i}}^{G}\bk{\Omega_{M_i,s,\chi}}$
%%	admits a unique irreducible subrepresentation. In the remaining cases, listed below, the length of the maximal semi-simple subrepresentation is $2$: 
%	All	$\pi = \Ind_{M_{i}}^{G}\bk{\Omega_{M_i,s,\chi}}$	admit a unique irreducible subrepresentation, with the exception of
%	$\Ind_{M_{2}}^{G}(\Omega_{M_{2},-\frac{1}{2},triv})$ and $\Ind_{M_{5}}^{G}(\Omega_{M_{5},-\frac{1}{2},triv})$, where the length of the maximal semi-simple subrepresentation is at most $2$ and the following representations, which admit a maximal semi-simple subrepresentation of length $2$:
%	\begin{enumerate}
%%		$[6,0,1]$, $[6,0,2]$, $[3,-1/2,2]$, $[1,-5/2,1]$
%		\item
%		$\Ind_{M_{1}}^{G}(\Omega_{M_{1},-\frac{5}{2},triv})$.
%		\item
%		$\Ind_{M_{3}}^{G}(\Omega_{M_{3},-\frac{1}{2},\chi})$,  where $ord(\chi)=2$.
%		\item
%		$\Ind_{M_{6}}^{G}(\Omega_{M_{6},0,\chi})$,  where $\chi^2=1$.
%		\item
%		$\Ind_{M_{7}}^{G}(\Omega_{M_{7},-\frac{3}{2},triv})$.
%	\end{enumerate}
%%	\todonum{Possibly also [2,-1/2,1] and [5,-1/2,1].}
%\end{enumerate}
%\end{Thm}

\begin{Remark}
	According to \cite{MR2123125}, the minimal representation of $E_8$ is the unique irreducible subrepresentation of $[8,-19/2,1]$ (Note that $P_8$ is the Heisenberg parabolic subgroup of $E_8$).
	From the data provided by Tadi\'c's reducibility criterion, it follows that the minimal representation is also a subquotient of the following cases: $[1,-17/2,1]$, $[2,-13/2,1]$ and $[4,-7/2,1]$.
	Indeed, it is isomorphic to their unique irreducible subrepresentations.
\end{Remark}

\begin{proof}
	We separate the proof into four parts: proof of reducibility (for all reducible $\pi$), proof of irreducibility (for most irreducible cases), proof of unique irreducible subrepresentation (for most cases) and exceptional cases.
	The last part is dealt with in \Cref{Sec:Exceptional_Cases} while the remainder of this section is devoted to the first three parts (which use the algorithm from \Cref{Subsec:The_Algorithm}).
	
	\subsection{Reducibility}
	\label{Subsec:Reducibility_Tables}
	
	For most reducible cases, it is enough to find $\pi'\neq \pi$ which shares an irreducible subquotient with $\pi$.
	As explained in \Cref{Subsubsec:Reducibility}, for most cases, one can find $\pi'=i_{M_i}^G\bk{\Omega_{M_j,t,\chi^l}}$ which satisfy the required conditions, while in others one needs to find $\pi'$ of the form $\pi'=i_{M_{j_1,j_2}}^G\Omega_{s_1,s_2,\chi,k_1,k_2}$, where $\Omega_{s_1,s_2,\chi,k_1,k_2}$ is given by \Cref{Eq:1_dim_rep_of_almost_max_Levi}.
%	, we denote the other case by $\coset{\coset{j_1,j_2}, \coset{s_1,s_2},\coset{k_1,k_2}}$.
%	\todonum{Check that this is we really put $k_1$ and $k_2$ and not $[ord(\chi_1),ord(\chi_2)]$}
	
	In the following tables we list, for each reducible $\coset{i,s,ord\bk{\chi}}$ a triple $\coset{j,t,ord\bk{\chi^l}}=\coset{j,t,ord\bk{\chi}}$ or $\coset{\coset{j_1,j_2}, \coset{s_1,s_2},\coset{k_1,k_2}}$ which provides a representation $\pi'\neq \pi$ sharing a common irreducible subquotient with $\pi$:
	
	\begin{itemize} 
		\item 
		For $\para{P}=\para{P}_{1}$ 
		\begin{center} 
			\footnotesize 
			\begin{longtable}{|c|c|c|} 
				\hline 
				\diagbox{$s$}{$ord\bk{\chi}$}  & $1$  & $2$  
				\\ \hline 
				$ -\frac{17}{2} $ & $ \left[8, -\frac{19}{2}, 1\right] $ & \notableentry 
				\\ \hline 
				$ -\frac{13}{2} $ & $ \left[8, -\frac{11}{2}, 1\right] $ & \notableentry 
				\\ \hline 
				$ -\frac{11}{2} $ & $ \left[8, -\frac{5}{2}, 1\right] $ & \notableentry 
				\\ \hline 
				$ -\frac{7}{2} $ & $ \left[\left[1, 8\right], \left[-\frac{7}{2}, -2\right], \left[0, 0\right]\right] $ & \notableentry 
				\\ \hline 
				$ -\frac{5}{2} $ & $ \left[3, -\frac{5}{2}, 1\right] $ & \notableentry 
				\\ \hline 
				$ -\frac{1}{2} $ & $ \left[7, -\frac{3}{2}, 1\right] $& $ \left[2, -\frac{5}{2}, 2\right] $ 
				\\ \hline 
				
				\caption{Data for the proof of the reducibility of $\Ind_{M_  1}^{G}(\Omega_{M_{1,s,\chi}})$} 
				\label{Tab::E7::COMP::P_1} 
			\end{longtable} 
		\end{center} 
		\item 
		For $\para{P}=\para{P}_{2}$ 
		\begin{center} 
			\footnotesize 
			\begin{longtable}{|c|c|c|} 
				\hline 
				\diagbox{$s$}{$ord\bk{\chi}$}  & $1$  & $2$   
				\\ \hline 
				$ -\frac{13}{2} $ & $ \left[8, -\frac{19}{2}, 1\right] $ & \notableentry 
				\\ \hline 
				$ -\frac{11}{2} $ & $ \left[8, -\frac{13}{2}, 1\right] $ & \notableentry  
				\\ \hline 
				$ -\frac{9}{2} $ & $ \left[8, -\frac{3}{2}, 1\right] $ & \notableentry  
				\\ \hline 
				$ -\frac{7}{2} $ & $ \left[1, -\frac{7}{2}, 1\right] $ & \notableentry  
				\\ \hline 
				$ -\frac{5}{2} $ & $ \left[7, -\frac{3}{2}, 1\right] $& $ \left[1, -\frac{1}{2}, 2\right] $ \\ \hline 
				$ -2 $  & \notableentry & \notableentry \\ \hline 
				$ -\frac{3}{2} $ & $ \left[\left[1, 8\right], \left[\frac{1}{2}, -2\right], \left[0, 0\right]\right] $& $ \left[\left[2, 8\right], \left[-2, 0\right], \left[1, 1\right]\right] $ \\ \hline 
				$ -\frac{1}{2} $ & $ \left[6, -1, 1\right] $& $ \left[6, -1, 2\right] $ \\ \hline 
				
				\caption{Data for the proof of the reducibility of $\Ind_{M_  2}^{G}(\Omega_{M_{2,s,\chi}})$} 
				\label{Tab::E7::COMP::P_2} 
			\end{longtable} 
		\end{center} 
		\item 
		For $\para{P}=\para{P}_{3}$ 
		\begin{center} 
			\footnotesize 
			\begin{longtable}{|c|c|c|c|} 
				\hline 
				\diagbox{$s$}{$ord\bk{\chi}$}  & $1$  & $2$  & $3$ \\ \hline 
				$ -\frac{11}{2} $ & $ \left[1, -\frac{19}{2}, 1\right] $ & \notableentry & \notableentry \\ \hline 
				$ -\frac{9}{2} $ & $ \left[8, -\frac{15}{2}, 1\right] $ & \notableentry & \notableentry \\ \hline 
				$ -\frac{7}{2} $ & $ \left[7, -\frac{9}{2}, 1\right] $ & \notableentry & \notableentry \\ \hline 
				$ -\frac{5}{2} $ & $ \left[7, -\frac{5}{2}, 1\right] $& $ \left[7, -\frac{5}{2}, 2\right] $ & \notableentry \\ \hline 
				$ -2 $ & $ \left[7, -1, 1\right] $& $ \left[7, -1, 2\right] $ & \notableentry \\ \hline 
				$ -\frac{3}{2} $ & $ \left[2, -\frac{3}{2}, 1\right] $& $ \left[\left[6, 7\right], \left[1, -\frac{7}{2}\right], \left[0, 1\right]\right] $ & \notableentry \\ \hline 
				$ -\frac{7}{6} $ & $ \left[2, -\frac{5}{6}, 1\right] $ & \notableentry& $ \left[\left[4, 5\right], \left[\frac{3}{2}, -\frac{19}{6}\right], \left[0, 1\right]\right] $ \\ \hline 
				$ -1 $ & $ \left[\left[6, 8\right], \left[-1, 0\right], \left[0, 0\right]\right] $& $ \left[\left[6, 7\right], \left[-3, 4\right], \left[0, 1\right]\right] $ & \notableentry \\ \hline 
				$ -\frac{1}{2} $ & $ \left[6, 0, 1\right] $& $ \left[\left[2, 3\right], \left[-\frac{5}{2}, 1\right], \left[1, 1\right]\right] $ & \notableentry \\ \hline 
				
				\caption{Data for the proof of the reducibility of $\Ind_{M_  3}^{G}(\Omega_{M_{3,s,\chi}})$} 
				\label{Tab::E7::COMP::P_3} 
			\end{longtable} 
		\end{center} 
		\item 
		For $\para{P}=\para{P}_{4}$ 
		\begin{center} 
			\tiny 
			\begin{longtable}{|c|c|c|c|c|c|} 
				\hline 
				\diagbox{$s$}{$ord\bk{\chi}$}  & $1$  & $2$  & $3$  & $4$  & $5$  \\ \hline 
				$ -\frac{7}{2} $ & $ \left[8, -\frac{19}{2}, 1\right] $ & \notableentry & \notableentry & \notableentry & \notableentry \\ \hline 
				$ -\frac{5}{2} $ & $ \left[7, -\frac{9}{2}, 1\right] $ & \notableentry & \notableentry & \notableentry & \notableentry \\ \hline 
				$ -2 $ & $ \left[7, -3, 1\right] $& $ \left[7, -3, 2\right] $ & \notableentry & \notableentry & \notableentry \\ \hline 
				$ -\frac{3}{2} $ & $ \left[6, -2, 1\right] $& $ \left[7, -\frac{1}{2}, 2\right] $ & \notableentry & \notableentry & \notableentry \\ \hline 
				$ -\frac{7}{6} $ & $ \left[6, -\frac{4}{3}, 1\right] $ & \notableentry& $ \left[\left[1, 6\right], \left[6, -\frac{10}{3}\right], \left[0, 2\right]\right] $ & \notableentry & \notableentry \\ \hline 
				$ -1 $ & $ \left[3, -1, 1\right] $& $ \left[3, -1, 2\right] $ & \notableentry & \notableentry & \notableentry \\ \hline 
				$ -\frac{5}{6} $ & $ \left[6, -\frac{1}{3}, 1\right] $ & \notableentry& $ \left[6, -\frac{1}{3}, 3\right] $ & \notableentry & \notableentry \\ \hline 
				$ -\frac{3}{4} $ & $ \left[3, -\frac{1}{4}, 1\right] $& $ \left[3, -\frac{1}{4}, 2\right] $ & \notableentry& $ \left[\left[5, 6\right], \left[-\frac{11}{4}, \frac{17}{4}\right], \left[3, 3\right]\right] $ & \notableentry 	\\ \hline 
				$ -\frac{1}{2} $ & $ \left[5, -\frac{1}{2}, 1\right] $& $ \left[\left[3, 8\right], \left[-\frac{1}{2}, \frac{3}{2}\right], \left[1, 0\right]\right] $& $ \left[\left[3, 8\right], \left[-\frac{1}{2}, \frac{3}{2}\right], \left[1, 1\right]\right] $& $ \left[\left[2, 6\right], \left[-1, 0\right], \left[3, 3\right]\right] $ & \notableentry \\ \hline 
				$ -\frac{3}{10} $ & $ \left[5, -\frac{1}{10}, 1\right] $ & \notableentry & \notableentry & \notableentry& $ \left[\left[2, 5\right], \left[-\frac{11}{10}, \frac{1}{5}\right], \left[2, 1\right]\right] $ \\ \hline 
				
				\caption{Data for the proof of the reducibility of $\Ind_{M_  4}^{G}(\Omega_{M_{4,s,\chi}})$} 
				\label{Tab::E7::COMP::P_4} 
			\end{longtable} 
		\end{center} 
		\item 
		For $\para{P}=\para{P}_{5}$ 
		\begin{center} 
			\footnotesize 
			\begin{longtable}{|c|c|c|c|c|} 
				\hline 
				\diagbox{$s$}{$ord\bk{\chi}$}  & $1$  & $2$  & $3$  & $4$ \\ \hline 
				$ -\frac{9}{2} $ & $ \left[8, -\frac{21}{2}, 1\right] $ & \notableentry & \notableentry & \notableentry \\ \hline 
				$ -\frac{7}{2} $ & $ \left[7, -\frac{11}{2}, 1\right] $ & \notableentry & \notableentry & \notableentry \\ \hline 
				$ -\frac{5}{2} $ & $ \left[6, -3, 1\right] $ & \notableentry & \notableentry & \notableentry \\ \hline 
				$ -2 $ & $ \left[1, -1, 1\right] $& $ \left[1, -1, 2\right] $ & \notableentry & \notableentry \\ \hline 
				$ -\frac{3}{2} $ & $ \left[2, -\frac{3}{2}, 1\right] $& $ \left[2, -\frac{3}{2}, 2\right] $ & \notableentry & \notableentry \\ \hline 
				$ -\frac{7}{6} $ & $ \left[2, -\frac{1}{6}, 1\right] $ & \notableentry& $ \left[2, -\frac{1}{6}, 3\right] $ & \notableentry \\ \hline 
				$ -1 $ & $ \left[6, -\frac{1}{2}, 1\right] $& $ \left[6, -\frac{1}{2}, 2\right] $ & \notableentry & \notableentry \\ \hline 
				$ -\frac{5}{6} $ & $ \left[3, -\frac{1}{6}, 1\right] $ & \notableentry& $ \left[\left[2, 3\right], \left[-\frac{7}{6}, 0\right], \left[2, 0\right]\right] $ & \notableentry \\ \hline 
				$ -\frac{1}{2} $ & $ \left[\left[3, 8\right], \left[-\frac{1}{2}, \frac{3}{2}\right], \left[0, 0\right]\right] $& $ \left[\left[3, 7\right], \left[-\frac{3}{2}, \frac{3}{2}\right], \left[1, 1\right]\right] $& $ \left[\left[3, 8\right], \left[-\frac{1}{2}, \frac{3}{2}\right], \left[1, 0\right]\right] $& $ \left[\left[2, 5\right], \left[\frac{1}{2}, -1\right], \left[3, 3\right]\right] $ \\ \hline 
				
				\caption{Data for the proof of the reducibility of $\Ind_{M_  5}^{G}(\Omega_{M_{5,s,\chi}})$} 
				\label{Tab::E7::COMP::P_5} 
			\end{longtable} 
		\end{center} 
		\item 
		For $\para{P}=\para{P}_{6}$ 
		\begin{center} 
			\footnotesize 
			\begin{longtable}{|c|c|c|c|} 
				\hline 
				\diagbox{$s$}{$ord\bk{\chi}$}  & $1$  & $2$  & $3$  \\ \hline 
				$ -6 $ & $ \left[8, -\frac{23}{2}, 1\right] $ & \notableentry & \notableentry \\ \hline 
				$ -5 $ & $ \left[7, -\frac{13}{2}, 1\right] $ & \notableentry & \notableentry \\ \hline 
				$ -4 $ & $ \left[8, -\frac{7}{2}, 1\right] $ & \notableentry & \notableentry \\ \hline 
				$ -3 $ & $ \left[1, -\frac{7}{2}, 1\right] $ & \notableentry & \notableentry \\ \hline 
				$ -\frac{5}{2} $ & $ \left[1, -2, 1\right] $& $ \left[1, -2, 2\right] $ & \notableentry \\ \hline 
				$ -2 $ & $ \left[7, -\frac{1}{2}, 1\right] $& $ \left[\left[2, 3\right], \left[-\frac{3}{2}, -1\right], \left[1, 0\right]\right] $ & \notableentry \\ \hline 
				$ -1 $ & $ \left[2, -\frac{1}{2}, 1\right] $& $ \left[2, -\frac{1}{2}, 2\right] $& $ \left[\left[1, 7\right], \left[-2, \frac{5}{2}\right], \left[1, 1\right]\right] $ \\ \hline 
				$ -\frac{1}{2} $ & $ \left[\left[2, 8\right], \left[-\frac{1}{2}, -1\right], \left[0, 0\right]\right] $& $ \left[\left[1, 2\right], \left[1, -\frac{3}{2}\right], \left[0, 1\right]\right] $ & \notableentry \\ \hline 
				$ 0 $ & $ \left[3, -\frac{1}{2}, 1\right] $& $ \left[\left[3, 7\right], \left[-\frac{3}{2}, \frac{1}{2}\right], \left[1, 1\right]\right] $ & \notableentry \\ \hline  
				\caption{Data for the proof of the reducibility of $\Ind_{M_  6}^{G}(\Omega_{M_{6,s,\chi}})$} 
				\label{Tab::E7::COMP::P_6} 
			\end{longtable}
		\end{center} 
		\item 
		For $\para{P}=\para{P}_{7}$ 
		\begin{center} 
			\footnotesize 
			\begin{longtable}{|c|c|c|} 
				\hline 
				\diagbox{$s$}{$ord\bk{\chi}$}  & $1$  & $2$  \\ \hline 
				$ -\frac{17}{2} $ & $ \left[8, -\frac{25}{2}, 1\right] $ & \notableentry \\ \hline 
				$ -\frac{11}{2} $ & $ \left[8, -\frac{11}{2}, 1\right] $ & \notableentry \\ \hline 
				$ -\frac{9}{2} $ & $ \left[8, -\frac{1}{2}, 1\right] $ & \notableentry \\ \hline 
				$ -\frac{5}{2} $ & $ \left[1, -\frac{5}{2}, 1\right] $& $ \left[3, -\frac{5}{2}, 2\right] $ \\ \hline 
				$ -\frac{3}{2} $ & $ \left[1, -\frac{1}{2}, 1\right] $ & \notableentry \\ \hline 
				$ -\frac{1}{2} $ & $ \left[\left[6, 8\right], \left[-\frac{7}{2}, \frac{15}{2}\right], \left[0, 0\right]\right] $ & $ \left[\left[3, 8\right], \left[-\frac{5}{2}, \frac{3}{2}\right], \left[1, 0\right]\right] $ \\ \hline 
				
				\caption{Data for the proof of the reducibility of $\Ind_{M_  7}^{G}(\Omega_{M_{7,s,\chi}})$} 
				\label{Tab::E7::COMP::P_7}
			\end{longtable}
		\end{center} 
		\item 
		For $\para{P}=\para{P}_{8}$ 
		\begin{center} 
			\footnotesize 
			\begin{longtable}{|c|c|} 
				\hline 
				\diagbox{$s$}{$ord\bk{\chi}$}  & $1$  \\ \hline 
				$ -\frac{19}{2} $ & $ \left[1, -\frac{17}{2}, 1\right] $ \\ \hline 
				$ -\frac{11}{2} $ & $ \left[1, -\frac{13}{2}, 1\right] $ \\ \hline 
				$ -\frac{1}{2} $ & $ \left[\left[1, 8\right], \left[-\frac{11}{2}, -1\right], \left[0, 0\right]\right] $ \\ \hline 
				\caption{Data for the proof of the reducibility of $\Ind_{M_  8}^{G}(\Omega_{M_{8,s,\chi}})$} 
				\label{Tab::E7::COMP::P_8}
			\end{longtable} 
		\end{center} 
	\end{itemize} 
	
	\subsection{Irreducibility}
	\label{SubSec:ProofAlgIrr}
	
	For most irreducible non-regular $\pi=i_{M_{i}}^{G}\bk{\Omega_{M_i,s,\chi}}$, it is possible to prove irreducibility using the branching rules calculation, as explained in \Cref{Subsubsec:Irreducibility}.
	The only exceptions were those of the cases $[4, 0, 1]$ and $[4, 0, 2]$, which will be dealt with in \Cref{Sec:Exceptional_Cases}.

	\subsection{Unique Irreducible Subrepresentation}
	\label{SubSec:ProofAlgUIS}
	
%	Here too, for most reducible non-regular $\pi=i_{M_{i}}^{G}\bk{\Omega_{M_i,s,\chi}}$, one can use the algorithm in \Cref{Subsubsec:UIS} to prove the existence of a unique irreducible subrepresentation.
	Here too, for most reducible non-regular $\pi=i_{M_{i}}^{G}\bk{\Omega_{M_i,s,\chi}}$, one can use the algorithm in \Cref{Subsubsec:UIS} to determine the length of the maximal semi-simple subrepresentation.
	There are, however, two kinds of exceptions:
	\begin{itemize}
		\item $\pi$ with unique irreducible subrepresentation which could not have been determined using the algorithm: $[2,-5/2,1]$, $[3,-3/2,1]$, $[3,-3/2,2]$, $[4,-1/2,1]$, $[4,-1/2,2]$, $[4,-3/2,1]$, $[5,-3/2,1]$, $[6,-2,1]$ and $[6,-2,2]$.
%		$[5,-1/2,1]$
		
		\item Non-unitary $\pi$ with a maximal semi-simple subrepresentation of length $2$:
		$[1, -5/2, 1]$, $[3, -1/2, 2]$ and $[7,-3/2,1]$.
%		, $[6, 0, 1]$ and $[6, 0, 2]$.
	\end{itemize}
%	\todonum{Add $\coset{2,-1/2,1}$ and $\coset{5,-1/2,1}$ when finished with them.}
	Also, there are the cases $\coset{2,-1/2,1}$ and $\coset{5,-1/2,1}$ where we are able to show that the maximal semi-simple subrepresentation is of length at most $2$.
	In these cases we further show that the irreducible spherical subquotient is a subrepresentation and describe the other candidate subrepresentation in terms of its Langlands data.
\end{proof}

\section{Exceptional Cases}
\label{Sec:Exceptional_Cases}

In this section, we finish the proof of \Cref{Thm:Main_theorem} by resolving the cases not solved by the algorithm of \Cref{Subsec:The_Algorithm}.

\subsection{Fully Resolved Cases}

In this subsection, we deal with all cases listed in \Cref{SubSec:ProofAlgIrr} and \Cref{SubSec:ProofAlgUIS} except for $\coset{2,-1/2,1}$ and $\coset{5,-1/2,1}$.
%This includes determining the reducibility or irreducibility for certain unitary degenerate principal series and lengths of certain degenerate principal series.
We organize the discussion according to similarities between the arguments used for the various cases, such as:

\begin{itemize}
	\item Using a chain of isomorphisms to embed $\pi$ in a parabolic induction that admits a unique irreducible subrepresentation in common with $\pi$.
	\item Using a chain of isomorphisms to embed $\pi$ in a parabolic induction with two irreducible subrepresentations and then calculating its intersection with $\pi$.
	\item Using $R$-groups of unitary principal series representations of Levi subgroups of $G$.
%	\item Using calculations in the Iwahori-Hecke algebra of $G$.
%	\todonum{Update this according to the last two cases.}
\end{itemize}

When $\chi=\Id_{F^\times}$, the representation $\pi=i_{M_i}^G\Omega_{i,s,\chi}$ admits a unique anti-dominant exponent and the \emph{Orthogonality Rule} in \cite[Appendix A.1]{SDPS_E7} implies that there exists a unique irreducible subquotient $\pi_0$ of $\pi$ such that $\lambda_{a.d.}\leq r_T^G\pi_0$.
However, when $ord\bk{\chi}>1$, this is not necessarily true as can be seen from \cite[Propositions 4.7 and 4.8]{SDPS_E7} or \Cref{Prop:cosocle_len_2_E7_P2_M1_O1} bellow, for example.
In the following lemma, we verify the existence of a unique such $\pi_0$ for certain $\pi=i_{M_i}^G\Omega_{i,s,\chi}$ with $ord\bk{\chi}=2$.
%The following proposition is useful for some of the following calculations.
\begin{Lem}
	\label{Lem:Unique_Subquot_with_ad_exp}
	Let $\chi$ be a character of $F^\times$ of order $2$.
	In the following cases, $\coset{i,s,ord\bk{\chi}}$, there exists a unique irreducible subquotient $\pi_0$ of $\pi=i_{M_i}^G\Omega_{i,s,\chi}$ such that $r_T^G\pi_0$ contains an anti-dominant exponent $\lambda_{a.d.}$ of $\pi$:
	\[
	\coset{3,-1/2,2}, 
	\coset{3,-3/2,2},
	\coset{4,-1/2,2},
	\coset{5,-1/2,2},
	\coset{6,-2,2} .
	\]
\end{Lem}

\begin{proof}
%	By \cite[Lemma 3.4]{HalawiHezi2020Tdps} and \cite[Equation (A.1)]{HalawiHezi2020Tdps}, this is equivalent to showing that $i_T^G\lambda_{a.d.}$ admits a unique irreducible subrepresentation.
	We argue by a similar argument to that of \cite[Proposition 4.7(3)]{SDPS_E7}.
	We do this by first studying the restricition $\bk{i_T^L\lambda_{a.d.}}\res{L^{der}}$ of $i_T^L\lambda_{a.d.}$ to $L^{der}$, where $L^{der}$ is the derived group of the Levi subgroup $L$ associated with
	\[
	\Theta\bk{\lambda_{a.d.}} = \set{\alpha\in\Delta \mvert \Real\bk{\inner{\lambda_{a.d.},\check{\alpha}}} = 0} .
	\]
	Since $G$ is simply-connected and simply-laced and $L$ is of type $A_{n_1}\times...\times A_{n_l}$, it follows that in these cases
	\[
	L^{der} = \prodl_{k=1}^l SL_{n_k}\bk{F}, \quad n_k\in\N.
	\]
	Hence, the length of $\bk{i_T^L\lambda_{a.d.}}\res{L^{der}}$ can be determined from \cite{MR620252}.
	In the following table, we list one anti-dimnant exponent $\lambda_{a.d.}$, the set $\Theta\bk{\lambda_{a.d.}}$ and $len\bk{\bk{i_T^L\lambda_{a.d.}}\res{L^{der}}}$ for each of these cases:
	\begin{center}
		\footnotesize
		\begin{longtable}[H]{|c|c|c|c|c|}
			\hline
			$\coset{i,s,ord\bk{\chi}}$ & $\lambda_{a.d.}$ & $\Theta\bk{\lambda_{a.d.}}$ & $L^{der}$ & $len\bk{\bk{i_T^L\lambda_{a.d.}}\res{L^{der}}}$ \\ \hline
			$\coset{3,-1/2,2}$ & $\eeightchar{\chi}{\chi}{\chi}{\chi}{-1}{\chi}{\chi}{-1}$ & $\set{1,2,3,4,6,7}$ & $SL_5\times SL_3$ & $1$ \\ \hline
			$\coset{3,-3/2,2}$ & $\eeightchar{\chi}{\chi}{\chi}{-1}{\chi}{\chi}{-1}{0}$ & $\set{1,2,3,5,6,8}$ & $SL_3\times SL_3\times SL_2\times SL_2$ & $2$ \\ \hline
			$\coset{4,-1/2,2}$ & $\eeightchar{\chi}{\chi}{\chi}{\chi}{-1}{\chi}{\chi}{0}$ & $\set{1,2,3,4,6,7,8}$ & $SL_5\times SL_4$ & $1$ \\ \hline
			$\coset{5,-1/2,2}$ & $\eeightchar{\chi}{\chi}{\chi}{\chi}{-1}{\chi}{\chi}{\chi}$ & $\set{1,2,3,4,6,7,8}$ & $SL_5\times SL_4$ & $1$ \\ \hline
			$\coset{6,-2,2}$ & $\eeightchar{\chi}{\chi}{\chi}{-1}{0}{\chi}{-1}{-1}$ & $\set{1,2,3,5,6}$ & $SL_3\times SL_3\times SL_2$ & $2$ \\ \hline
		\caption{Data for the proof of \Cref{Lem:Unique_Subquot_with_ad_exp}.}
		\label{Table:Data_Lem:Unique_Subquot_with_ad_exp}
		\end{longtable}
%		\captionof{table}{Data for the proof of \Cref{Lem:Unique_Subquot_with_ad_exp}.}\label{Table:Data_Lem:Unique_Subquot_with_ad_exp}
	\end{center}
	In the cases $\coset{3,-1/2,2}$, $\coset{4,-1/2,2}$ and $\coset{5,-1/2,2}$ the representation $\bk{i_T^L\lambda_{a.d.}}\res{L^{der}}$ is irreducible and hence, so is $i_T^L\lambda_{a.d.}$.
	In the other two cases, $\coset{3,-3/2,2}$ and $\coset{6,-2,2}$, the representation $\bk{i_T^L\lambda_{a.d.}}\res{L^{der}}$ has length $2$.
	Since $L$ is generated by $L^{der}$ and $T$, the reducibility of $i_T^L\lambda_{a.d.}$ is determined by the action of $T$ on the pieces of $\bk{i_T^L\lambda_{a.d.}}\res{L^{der}}$.
	Indeed, it follow that $i_T^L\lambda_{a.d.}$ is irreducible (see \cite[Proposition 4.7(3)]{SDPS_E7} for more details).

	We note that
	\[
	\Real\bk{\inner{\lambda,\check{\alpha}}} < 0 \quad \forall \alpha\notin\Delta_\lambda .
	\]
	Hence, by Langlands' unique irreducible subrepresentation theorem (see \cite[Section 1]{MR1665057}),
	\[
	i_T^G \lambda_{a.d.} \cong i_L^G\bk{i_T^L\lambda_{a.d.}}
	\]
	admits a unique irreducible subrepresentation $\pi_0$ which appears in $i_T^G\lambda_{a.d.}$ with multiplicity $1$.
	We now argue that
	\[
	mult\bk{\lambda_{a.d.},r_T^G \pi_0} = mult\bk{\lambda_{a.d.},r_T^G \bk{i_T^G\lambda_{a.d.}}} .
	\]	
	Assume to the contrary that there exists a subquotient $\tau\neq\pi_0$ of $i_T^G\lambda_{a.d.}$ such that $\lambda_{a.d.}\leq r_T^G\bk\tau$.
	By a central character argument, see \cite[Lemma. 3.12]{SDPS_E6}, it holds that
	$\tau\hookrightarrow i_T^G\lambda_{a.d.}$.
	Since $\pi_0$ is the unique irreducible subrepresentation of $i_T^G\lambda_{a.d.}$, it follows that $\tau\cong\pi_0$, contradicting the multiplicity $1$ property of $\pi_0$.
\end{proof}

\begin{Remark}
	If one chooses a different anti-dominant element $\lambda_{a.d.}'$ in the orbit of $\lambda_{a.d.}$ in \Cref{Table:Data_Lem:Unique_Subquot_with_ad_exp}, they would have the same real part.
	Thus $\Theta\subset \Delta$ and $L^{der}$ are invariant under the choice of representative of the orbit.
	Further more, $\bk{i_T^L\lambda_{a.d.}'}\res{L^{der}}$ is isomorphic to $\bk{i_T^L\lambda_{a.d.}}\res{L^{der}}$
\end{Remark}

%\begin{Remark}
%	Note that, when $\chi=\Id_T$, the uniqueness of $\pi_0$ is automatic.
%	However, when $ord\bk{\chi}>1$, this is not necessarily true as can be seen from \Cref{Prop:cosocle_len_2_E7_P2_M1_O1} for example.
%\end{Remark}

\begin{Prop}
	In the following cases, $\coset{i,s,ord\bk{\chi}}$, the representation $\pi=i_{M_i}^G\Omega_{i,s,\chi}$ admits a unique irreducible subrepresentation:
	\[
	\coset{3,-3/2,2}, \coset{6,-2,2}
	\]
\end{Prop}

\begin{proof}
	Let $\lambda_{a.d.}$ be an anti-dominant exponent of $\pi$ and $\pi_0$ be the unique irreducible subquotient of $\pi$ such that $\lambda_{a.d.}\leq r_T^G\pi_0$ as in \Cref{Table:Data_Lem:Unique_Subquot_with_ad_exp} and \Cref{Lem:Unique_Subquot_with_ad_exp}.
	One checks, using a branching rule calculation, that in these cases
	\[
	mult\bk{\lambda_{a.d.},r_T^G\pi_0} = \Card{\Stab_W\bk{\lambda_{a.d.}}} \quad \Rightarrow \quad
	mult\bk{\lambda_0,r_T^G\pi_0} = mult\bk{\lambda_0,r_T^G\pi} = 2.
	\]
	From which the claim follows.
	Here,
	\[
	\Card{\Stab_W\bk{\lambda_{a.d.}}} = \piece{8, & \coset{i,s,ord\bk{\chi}}=\coset{3,-3/2,2} \\ 4,& \coset{i,s,ord\bk{\chi}}=\coset{6,-2,2} .}
	\]
\end{proof}

In the following, we make a repeated use of a certain inclusion argument.
It is thus convenient to state it in general at this point.
For that matter, let $\Omega_0=\Omega_{i,s,\chi}$ and $\pi=i_{M_i}^G\Omega_0$.
For $j\neq i$ let $M_{i,j}=M_i\cap M_j$ and $\Omega_1=r_{M_{i,j}}^{M_i}\Omega_0$.
By Frobenius reciprocity,
\[
\Omega_0\hookrightarrow i_{M_{i,j}}^{M_i}\Omega_1 .
\]
and hence, by induction in stages,
\begin{equation}
\label{Eq:Jacquet_and_Inclusion}
\pi=i_{M_i}^G\Omega_0 \hookrightarrow i_{M_i}^G\bk{i_{M_{i,j}}^{M_i}\Omega_1} \cong i_{M_{i,j}}^G \Omega_1 \cong i_{M_j}^G\bk{i_{M_{i,j}}^{M_j}\Omega_1} .
\end{equation}

%In the following we make a repeated use of two applications of Frobenius reciprocity and induction in stages.
%It is thus convenient to state them in general at this point.
%For that matter, let $\Omega_0=\Omega_{i,s,\chi}$ and $\pi=i_{M_i}^G\Omega_0$.
%\begin{enumerate}
%	\item Let $j\neq i$, $M_{i,j}=M_i\cap M_j$ and $\Omega_1=r_{M_{i,j}}^{M_i}\Omega_0$.
%	Then, by Frobenius reciprocity,
%	\[
%	\Omega_0\hookrightarrow i_{M_{i,j}}^{M_i}\Omega_1 .
%	\]
%	Hence, by induction in stages,
%	\begin{equation}
%	\label{Eq:Jacquet_and_Inclusion}
%	\pi=i_{M_i}^G\Omega_0 \hookrightarrow i_{M_i}^G\bk{i_{M_{i,j}}^{M_i}\Omega_1} \cong i_{M_{i,j}}^G \Omega_1 \cong i_{M_j}^G\bk{i_{M_{i,j}}^{M_j}\Omega_1} .
%	\end{equation}
%	
%	\item Assume that $i_{M_{i,j}}^{M_j}\Omega_1$ is irreducible, then $i_{M_{i,j}}^{M_j}\Omega_1\cong i_{M_{k,j}}^{M_j}\Omega_2$ for some $k\neq j$.
%	Here $i_{M_{k,j}}^{M_j}\Omega_2$ is the invert representation of $i_{M_{i,j}}^{M_j}\Omega_1$.
%\end{enumerate}

\begin{Prop}
	\label{Prop:Unique_Irr_Subs_Jac_and_Ind}
	In the following cases, $\coset{i,s,ord\bk{\chi}}$, the representation $\pi=i_{M_i}^G\Omega_{i,s,\chi}$ admits a unique irreducible subrepresentation:
	\[
	\coset{2,-5/2,1},
	\coset{3,-3/2,1},
	\coset{4,-1/2,1},
	\coset{4,-3/2,1},
	\coset{5,-1/2,2}
	\coset{5,-3/2,1},
	\coset{6,-2,1} .
	\]
\end{Prop}

\begin{proof}
	Since the proof is similar for all cases, we start by outlining it and conclude by listing the relevant data for each of the cases in \Cref{Table:Data_Prop:Unique_Irr_Subs_Jac_and_Ind}.
	
	For $\coset{i,s,ord\bk{\chi}}$ as above, let $\lambda_0=r_T^G\Omega_{i,s,\chi}$.
	For $j\neq i$ as in \Cref{Table:Data_Prop:Unique_Irr_Subs_Jac_and_Ind}, let $M_{i,j}=M_i\cap M_j$ and $\Omega_1=r_{M_{i,j}}^{M_i}\Omega_0$.
	It holds that
	\[
	\lambda_0=r_T^{M_i}\Omega=r_T^{M_{i,j}}\Omega_1.
	\]
	Assume that $i_{M_{i,j}}^{M_j}\Omega_1$ is an irreducible representation of $M_j$, this can be verified by a branching rule calculation in $M_j$ or by referring to previous knowledge on the representation theory of $M_j$ as explained bellow in \Cref{Remark:Irreducibility_of_representations_of_Levi}.
	Let $\lambda_1$ be an $M_j$-anti-dominant exponent of $i_{M_{i,j}}^{M_j}\Omega_1$, that is $\lambda_1\leq r_T^{M_j}\bk{i_{M_{i,j}}^{M_j}\Omega_1}$ and
	\begin{equation}
	\label{Eq:M_ij_anti_dominant}
	\Real\bk{\inner{\lambda_1,\check{\alpha_l}}} \leq 0 \quad \forall l\neq j .
	\end{equation}
	By a central character argument, see \cite[Lemma. 3.12]{SDPS_E6}, it holds that
	\[
	i_{M_{i,j}}^{M_j}\Omega_1 \hookrightarrow i_T^{M_j}\lambda_1 .
	\]
	Hence, by \Cref{Eq:Jacquet_and_Inclusion} and induction in stages,
	\[
	\pi \hookrightarrow i_{M_j}^G\bk{i_{M_{i,j}}^{M_j}\Omega_1} \hookrightarrow i_T^G\lambda_1 .
	\]
	By Frobenius reciprocity, any irreducible subrepresentation $\tau$ of $i_T^G\lambda_1$ must satisfy
	\[
	\lambda_1\leq r_T^G\tau .
	\]
	
	On the other hand, let $\lambda_{a.d.}$ be an anti-dominant exponent of $\pi$ and let $\pi_0$ denote the unique irreducible subquotient of $\pi$ such that $\lambda_{a.d.}\leq r_T^G\pi_0$.
	If $ord\bk{\chi}=1$ the existence of $\pi_0$ is automatic, and if $ord\bk{\chi}=2$, it follows from \Cref{Lem:Unique_Subquot_with_ad_exp}.
	
	In each of these cases, a branching rule calculation implies that
	\[
	mult\bk{\lambda_1,r_T^G\pi} = mult\bk{\lambda_1,r_T^G\pi_0} .
	\]
	It follows that $\pi_0$ is the unique irreducible subquotient of $\pi$ which admits $\lambda_1$ as an exponent.
	In other words, $\pi_0$ is the unique irreducible subquotient of $\pi$ such that
	\[
	\pi_0\hookrightarrow i_T^G\lambda_1 .
	\]
	On the other hand, since
	\[
%	mult\bk{\lambda_1,r_T^G\pi} 
%	= mult\bk{\lambda_1,r_T^G i_T^G\lambda_1}
%	= mult\bk{\lambda_1,r_T^G\pi_0} ,
	mult\bk{\lambda_{a.d.},r_T^G\pi} 
	= mult\bk{\lambda_{a.d.},r_T^G i_T^G\lambda_1}
	= mult\bk{\lambda_{a.d.},r_T^G\pi_0} ,	
	\]
	it follows that $\pi_0$ appears in $i_T^G\lambda_1$, and in $\pi$, with multiplicity $1$.
	We conclude that $\pi_0$ is the unique irreducible subrepresentation of $\pi$.
	
	\begin{center}
		\footnotesize
		\begin{longtable}[H]{|c|c|c|c|c|c|c|}
			\hline
			$i$ & $s$ & $ord\bk{\chi}$ & $j$ & $\lambda_0$ & $\lambda_1$ & $\lambda_{a.d.}$ \\ \hline %\coset{D_7,7,0,1}
			$2$ & $-\frac{5}{2}$ & $1$ & $1$ & $\eeightchar{-1}{5}{-1}{-1}{-1}{-1}{-1}{-1}$ & $\eeightchar{2}{-1}{-1}{0}{-1}{0}{-1}{0}$ & $\eeightchar{-1}{0}{0}{-1}{0}{0}{-1}{0}$
			\\ \hline %[E_6,5,-1/2,1]\times triv_{A_1}
			$3$ & $-\frac{3}{2}$ & $1$ & $7$ & $\eeightchar{-1}{-1}{4}{-1}{-1}{-1}{-1}{-1}$ & $\eeightchar{0}{0}{0}{-1}{0}{-1}{1}{-1}$ & $\eeightchar{0}{0}{0}{-1}{0}{0}{-1}{0}$ \\ \hline %\coset{D_5,3,0,1}\times triv_{A_2}
			$4$ & $-\frac{1}{2}$ & $1$ & $8$ & $\eeightchar{-1}{-1}{-1}{3}{-1}{-1}{-1}{-1}$ &
			$\eeightchar{0}{0}{0}{0}{-1}{0}{0}{0}$ & $\eeightchar{0}{0}{0}{0}{-1}{0}{0}{0}$ \\ \hline % \coset{E_7,4,0,1}
			$4$ & $-\frac{3}{2}$ & $1$ & $6$ & $\eeightchar{-1}{-1}{-1}{2}{-1}{-1}{-1}{-1}$ & $\eeightchar{0}{0}{0}{-1}{0}{0}{-1}{-1}$ & $\eeightchar{0}{0}{0}{-1}{0}{0}{-1}{-1}$ \\ \hline
			$5$ & $-\frac{1}{2}$ & $2$ & $8$ & $\eeightchar{-1}{-1}{-1}{-1}{4+\chi}{-1}{-1}{-1}$ & $\eeightchar{\chi}{\chi}{\chi}{-1+\chi}{\chi}{\chi}{\chi}{1+\chi}$ & $\eeightchar{\chi}{\chi}{\chi}{\chi}{-1}{\chi}{\chi}{\chi}$
			\\ \hline
			$5$ & $-\frac{3}{2}$ & $1$ & $7$ & $\eeightchar{-1}{-1}{-1}{-1}{3}{-1}{-1}{-1}$ & $\eeightchar{0}{0}{0}{-1}{0}{-1}{1}{-1}$ & $\eeightchar{0}{0}{0}{-1}{0}{0}{-1}{0}$ \\ \hline %\coset{E_6,5,-1/2,1}\times triv_{A_1}
			$6$ & $-2$ & $1$ & $8$ & $\eeightchar{-1}{-1}{-1}{-1}{-1}{4}{-1}{-1}$ & $\eeightchar{0}{0}{0}{-1}{0}{-1}{-1}{2}$ & $\eeightchar{0}{0}{0}{-1}{0}{0}{-1}{-1}$ \\ \hline %\coset{E_7,6,-3/2,1}
		\caption{Data for the proof of \Cref{Prop:Unique_Irr_Subs_Jac_and_Ind}.}
		\label{Table:Data_Prop:Unique_Irr_Subs_Jac_and_Ind}
		\end{longtable}
%		\captionof{table}{Data for the proof of \Cref{Prop:Unique_Irr_Subs_Jac_and_Ind}.}\label{Table:Data_Prop:Unique_Irr_Subs_Jac_and_Ind}
	\end{center}
	For the case $\coset{5,-1/2,2}$, $\chi$ denotes in \Cref{Table:Data_Prop:Unique_Irr_Subs_Jac_and_Ind} a character of $F^\times$ of order $2$.
	
\end{proof}

\begin{Remark}
	\label{Remark:Irreducibility_of_representations_of_Levi}
	The above argument, as well as other arguments bellow, used the fact that $\sigma=i_{M_{i,j}}^{M_j}\Omega_1$ is an irreducible representation of $M_j$.
	One way to prove the irreducibility of $\sigma$ is using branching rule calculations but it can also be inferred from previous works.
	We wish to allow for a common reference for following arguments and thus we treat all values of $j$ and not only $1$, $6$, $7$ and $8$.
	
	More precisely, $\sigma$ is a degenerate principal series of $M_j$ and if its restriction $\sigma\res{M_j^{der}}$ to $M_j^{der}$ is irreducible, then so is $\sigma$.
	For different values of $j$, one can find the list of reducible and irreducible degenerate principal series of $M_j^{der}$ in the following sources:
	\begin{itemize}
		\item \cite{MR2017065} for $j=1$.
		\item \cite{MR1141803} for $j=2,3,4,5$
		\item \cite{MR2017065} and \cite{MR1141803} for $j=6$.
		%		Here the irreducibility is proven for the special orthogonal group which is not simply-connected, the result for $M_6^{der}$ follows by applying the results of \cite{MR1141803} (one should note that $M_6^{der}=Spin_{10}\times SL_2$ and $Spin_{10}$ is an algebraic cover of a finite index subgroup of $SO\bk{10}$).
		%		The irreducibility of the $M_6$-representation will then follow.
		\item \cite{SDPS_E6,SDPS_E7} for $j=7,8$.
		%		Here the irreducibility is proven for the derived group $M_j^{der}$, but they follow easily for $M_j$.
	\end{itemize}
	It should be noted here that \cite{MR2017065} deals only with representations of orthogonal groups, while $M_1^{der}=Spin_{14}$ and $M_6^{der}=Spin_{10}\times SL_2$.
	That is, in order to prove irreducibility, one should also take the isogeny map into account, using  \cite{MR1141803}.

%	However, it should be noted that these sources shows the irreducibility of $i_{M_{i,j}}^{M_j}\Omega_1$ as a representations of $M_j^{der}$ from which one can deduce the irreducibility as a representation of $M_j$.
%	In fact, the results of \cite{MR2017065} apply only to representations of standard orthogonal groups while $M_1^{der}=Spin_{14}$ and $M_6^{der}=Spin_{10}\times SL_2$.
%	In order to lift these results to $M_j^{der}$ on uses \cite{MR1141803}.
%	The irreducibility of $i_{M_{i,j}}^{M_j}\Omega_1$ as a representation of $M_j$ would then follow.
\end{Remark}

A slight modification of the argument made in \Cref{Prop:Unique_Irr_Subs_Jac_and_Ind} can be used to prove:
\begin{Prop}
	The representation $\pi=i_{M_i}^G\Omega_{4,-\frac{1}{2},\chi}$, where $\chi$ has order $2$, admits a unique irreducible subrepresentation.
\end{Prop}

\begin{proof}
	Let $\Omega_0=\Omega_{4,-\frac{1}{2},\chi}$ and let
	\[
	\lambda_0 = r_T^{M_4}\Omega_0 = \eeightchar{-1}{-1}{-1}{3+\chi}{-1}{-1}{-1}{-1} .
	\]
	We fix an anti-dominant exponent
	\[
	\lambda_{a.d.} = \eeightchar{\chi}{\chi}{\chi}{\chi}{-1}{\chi}{\chi}{0}
	\]
	of $\pi$.
	Let $\Omega_1=r_{M_{3,4}}^{M_4}\Omega_0$.
	By \Cref{Eq:Jacquet_and_Inclusion},
	\[
	\pi \hookrightarrow i_{M_3}^G \bk{i_{M_{3,4}}^{M_3}\Omega_1} .
	\]
	The representation $i_{M_{3,4}}^{M_3}\Omega_1$ of $M_3$ is an irreducible degenerate principal series of $M_3$.
	This can be verified by a branching rule calculation in $M_7$ or by a similar argument to that of \Cref{Remark:Irreducibility_of_representations_of_Levi}. %\coset{E_6,3,-1/2,1}\times triv_{A_1}
	Furthermore, it holds that $i_{M_{3,4}}^{M_3}\Omega_1 \cong i_{M_{3,7}}^{M_3}\Omega_2$, %triv_{A_1}\times \coset{A_6,2,1/2,2}
	where
	\[
	r_T^{M_{3,7}}\Omega_2 = \eeightchar{-1}{-1}{3}{-1}{-1}{-1}{2+\chi}{-1} 
	\]
	and $i_{M_{3,7}}^{M_3}\Omega_2$ is the invert of the irreducible $i_{M_{3,4}}^{M_3}\Omega_1$.
	It follows, by induction in stages, that
	\begin{equation}
		\label{Eq:Injection_In_proof_of_UIS_4_1/2_2}
		\pi \hookrightarrow i_{M_3}^G \bk{i_{M_{3,7}}^{M_3}\Omega_2} \cong  i_{M_7}^G \bk{i_{M_{3,7}}^{M_7}\Omega_2}.
	\end{equation}
%	We, again, point out that $i_{M_{3,7}}^{M_7}\Omega_2$ is irreducible.
	The $M_7$-anti-dominant exponent of $i_{M_{3,7}}^{M_7}\Omega_2$ (see \Cref{Eq:M_ij_anti_dominant}) is given by
	\[
	\lambda_1=\eeightchar{-1}{0}{0}{-1}{0}{0}{2+\chi}{-1} .
	\]
	By a central character argument, see \cite[Lemma. 3.12]{SDPS_E6},
	\[
	i_{M_{3,7}}^{M_7}\Omega_1 \hookrightarrow i_T^{M_7}\lambda_1 .
	\]
	and hence
	\[
	\pi\hookrightarrow i_T^G\lambda_1 .
	\]
	On the other hand, for the unique irreducible subrepresentation $\pi_0$ of $\pi$ such that $\lambda_{a.d.}\leq r_T^G\pi_0$ as in \Cref{Lem:Unique_Subquot_with_ad_exp}, a branching rule calculation shows that
	\[
	mult\bk{\lambda_1,r_T^G\pi} 
%	= mult\bk{\lambda_1,r_T^G i_T^G\lambda_1}
	= mult\bk{\lambda_1,r_T^G\pi_0} = 48.
	\]
	Arguing as in \Cref{Prop:Unique_Irr_Subs_Jac_and_Ind}, it follows that $\pi$ admits a unique irreducible subrepresentation.	
\end{proof}

\begin{Prop}
	\label{Prop:Unitary_cases}
	The representation $\pi=i_{M_4}^G\Omega_{4,0,\chi}$, where $\chi^2=1$, is irreducible.
\end{Prop}

\begin{proof}
	Let $\Omega_0=\Omega_{4,0,\chi}$ and let $\Omega_1=r_{M_{2,4}}^{M_4}\Omega_1$.
	By \Cref{Eq:Jacquet_and_Inclusion},
	\[
	\pi \hookrightarrow i_{M_2}^G \bk{i_{M_{2,4}}^{M_2}\Omega_1} .
	\]
%	We write a similar sequence of isomorphisms as in \Cref{Eq:Sequence_of_isomorphisms}, summarized in the following table (analog to \Cref{Table:Data_Prop:P7_M3_2_O1}).
%	By \Cref{Eq:Jacquet_and_Inclusion},
%	\[
%	\pi \hookrightarrow i_{M_1}^G\bk{i_{M_{1,7}}^{M_1}\Omega_1} .
%	\]
	In what follows, we consider a sequence of isomorphisms, similar to the one in \Cref{Eq:Injection_In_proof_of_UIS_4_1/2_2}.
	That is, we wish to write a sequence of isomorphisms:
	\begin{equation}
		\label{Eq:Sequence_of_isomorphisms_G_level}
		i_{M_{i_k}}^G\bk{i_{M_{i_k,j_k}}^{M_{i_k}}\Omega_k} \cong	i_{M_{j_{k+1}}}^G\bk{i_{M_{i_{k+1},j_{k+1}}}^{M_{j_{k+1}}}\Omega_{k+1}} \cong
		i_{M_{i_{k+1}}}^G\bk{i_{M_{i_{k+1},j_{k+1}}}^{M_{i_{k+1}}}\Omega_{k+1}} ,
	\end{equation}
	where
	\begin{itemize}
		\item $j_{k+1}=i_k$.
		\item $\Omega_k$ is a $1$-dimensional representation of $M_{i_k,j_k}$.
		\item $\Omega_{k+1}$ is a $1$-dimensional representation of $M_{i_k,j_{k+1}}$.
		\item $i_{M_{i_k,j_k}}^{M_{i_k}}\Omega_k$ is an irreducible representation of $M_{i_k}$.
		\item $i_{M_{j_{k+1},i_{k+1}}}^{M_{j_{k+1}}}\Omega_{k+1}$ is the invert representation of $i_{M_{i_k,j_k}}^{M_{i_k}}\Omega_k$.
	\end{itemize}
	
	The sequence in \Cref{Eq:Sequence_of_isomorphisms_G_level} relies on an isomorphism
	\begin{equation}
		\label{Eq:Sequence_of_isomorphisms}
		i_{M_{i_k,j_k}}^{M_{i_k}}\Omega_k \cong i_{M_{j_{k+1},i_{k+1}}}^{M_{j_{k+1}}}\Omega_{k+1}
	\end{equation}
	of irreducible representations for each $k$.
	The irreducibility of each of the $i_{M_{i_k,j_k}}^{M_{i_k}}\Omega_k$ follows from a branching rule calculation in $M_{i_k}$ and it can also be inferred from the references in \Cref{Remark:Irreducibility_of_representations_of_Levi}.
	
%	$j_{k+1}=i_k$, $\Omega_k$ is a $1$-dimensional representation of $M_{i_k,j_k}$, $\Omega_{k+1}$ is a $1$-dimensional representation of $M_{i_k,j_{k+1}}$, $i_{M_{i_k,j_k}}^{M_{i_k}}\Omega_k$ is an irreducible representation of $M_{i_k}}$ and $i_{M_{j_{k+1},i_{k+1}}}^{M_{j_{k+1}}}\Omega_{k+1}$ is the invert representation of $i_{M_{i_k,j_k}}^{M_{i_k}}\Omega_k$.
	We summarize this data for the sequence of isomorphisms in the following table
	\begin{center}
		\footnotesize
		\begin{longtable}{|c||c|c|c|c|c|}\hline 
			$k$ & $i_k$ & $j_k$ & $r_T^{M_{i_k.j_k}}\Omega_k$ & $i_{k+1}$ & $r_T^{M_{j_{k+1},i_{k+1}}}\Omega_{k+1}$ \\ \hline
			$1$ & $2$ & $4$ & $\eeightchar{-1}{-1}{-1}{\frac{7}{2}+\chi}{-1}{-1}{-1}{-1}$ & $6$ & $\eeightchar{-1}{\frac{7}{2}+\chi}{-1}{-1}{-1}{\frac{5}{2}+\chi}{-1}{-1}$ \\ \hline
			$2$ & $6$ & $2$ & $\eeightchar{-1}{\frac{7}{2}+\chi}{-1}{-1}{-1}{\frac{5}{2}+\chi}{-1}{-1}$ & $5$ & $\eeightchar{-1}{-1}{-1}{-1}{\frac{5}{2}+\chi}{\frac{3}{2}+\chi}{-1}{-1}$ \\ \hline
			$3$ & $5$ & $6$ & $\eeightchar{-1}{-1}{-1}{-1}{\frac{5}{2}+\chi}{\frac{3}{2}+\chi}{-1}{-1}$ & $8$ & $\eeightchar{-1}{-1}{-1}{-1}{4}{-1}{-1}{\frac{1}{2}+\chi}$ \\ \hline
		\caption{Data for the proof of \Cref{Prop:Unitary_cases}.}
		\label{Table:Data_Prop:Unitary_cases}
		\end{longtable}
%		\captionof{table}{Data for the proof of \Cref{Prop:Unitary_cases}.}\label{Table:Data_Prop:Unitary_cases}
	\end{center}
	
	Namely, the following sequence of isomorphisms hold
	\begin{equation}
		\label{Eq:Sequence_of_isomorphisms_unitary_case}
		\begin{array}{l}
			i_{M_2}^G\bk{i_{M_{2,4}}^{M_2}\Omega_1}
			\cong i_{M_2}^G\bk{i_{M_{2,6}}^{M_2}\Omega_1} 
			\cong i_{M_6}^G\bk{i_{M_{2,6}}^{M_6}\Omega_1} \\
			\cong i_{M_6}^G\bk{i_{M_{5,6}}^{M_6}\Omega_1}
			\cong i_{M_5}^G\bk{i_{M_{5,6}}^{M_5}\Omega_1}
			\cong i_{M_8}^G\bk{i_{M_{5,8}}^{M_8}\Omega_4}.
		\end{array}
	\end{equation}
	In each step we use either induction in stages or invert a representation.
	That is:
	\begin{itemize}
		\item $i_{M_2}^G\bk{i_{M_{2,4}}^{M_2}\Omega_1}
		\cong i_{M_2}^G\bk{i_{M_{2,6}}^{M_2}\Omega_1}$ since $i_{M_{2,4}}^{M_2}\Omega_1$ and $i_{M_{2,6}}^{M_2}\Omega_1$ are two irreducible degenerate principal series representations of $M_2$ which invert to each other.
		
		\item $i_{M_2}^G\bk{i_{M_{2,6}}^{M_2}\Omega_1} 
		\cong i_{M_6}^G\bk{i_{M_{2,6}}^{M_6}\Omega_1}$ by induction in stages.
		\item etc.
	\end{itemize}

%	Namely, it holds that
%	\begin{equation}
%	\label{Eq:Sequence_of_isomorphisms_unitary_case}
%	i_{M_2}^G\bk{i_{M_{2,4}}^{M_2}\Omega_1} \cong ... \cong i_{M_8}^G\bk{i_{M_{5,8}}^{M_8}\Omega_4} .
%	\end{equation}
	It follows that
	\[
	\pi \hookrightarrow i_{M_8}^G\bk{i_{M_{5,8}}^{M_8}\Omega_4} .
	\]
	Let
	\[
	\lambda_1 = \eeightchar{0}{0}{0}{-1}{0}{0}{0}{\frac{5}{2}+\chi}
	\]
	be the $M_8$-anti-dominant exponent of $i_{M_8}^G\bk{i_{M_{5,8}}^{M_8}\Omega_4}$.
	By a central character argument, see \cite[Lemma. 3.12]{SDPS_E6}, it follows that
	\[
	\pi \hookrightarrow i_{M_8}^G\bk{i_{M_{5,8}}^{M_8}\Omega_4} \hookrightarrow i_T^G\lambda_1 .
	\]
	Let
	\[
	\lambda_{a.d.} = \eeightchar{0}{0}{-\frac{1}{2}+\chi}{0}{0}{-\frac{1}{2}+\chi}{0}{0}
	\]
	denote the anti-dominant exponent of $\pi$ and let $\pi_0$ denote the unique (due to \cite[Lemma A.1]{SDPS_E7}) subquotient of $\pi$ such that $\lambda_{a.d.}\leq r_T^G\pi_0$.
%	It is unique due to \cite[Lemma A.1]{HalawiHezi2019TDPS} and it is a subrepresentation due to the fact that $\pi$ is semi-simple.
	A branching rule calculation yields that
	\[
	mult\bk{\lambda_1,r_T^G\pi_0} = 288 = mult\bk{\lambda_1,r_T^G\pi}
	\]
	Hence, $\pi_0$ is the unique subquotient of $\pi$ which is a subrepresentation of $i_T^G\lambda_1$.
	Thus, $\pi_0$ is the unique irreducible subrepresentation of $\pi$.
	Since $\pi$ is semi-simple, it follows that it is irreducible.
\end{proof}

\begin{Prop}
	\label{Prop:cosocle_len_2_E7_P2_M1_O1}
	In the following cases, $\coset{i,s,ord\bk{\chi}}$, the maximal semi-simple subrepresentation  of $\pi=i_{M_i}^G\Omega_{i,s,\chi}$ has length $2$:
	\[
	\coset{1,-5/2,1}, \coset{3,-1/2,2} .
	\]
\end{Prop}

\begin{proof}
	We prove the two cases using a similar argument and so we deal with them simultaneously and point out where the arguments diverge.
	Let $\Omega_0=\Omega_{i,s,\chi}$ and $\pi=i_{M_i}^G\Omega_0$.
	We point out that
	\[
	mult\bk{\lambda_0,r_T^G\pi} = 2 ,
	\]
	where $\lambda_0=r_T^{M_i}\Omega_0$.
	Hence, $\pi$ admits a maximal semi-simple subrepresentation of length at most $2$.
	
%	\sepline
%
%	Let $\Omega_1 = r_{M_{i,j}}^{M_i}\Omega_0$, where
%	\[
%	j = \piece{2, & i=1 \\ 1,& i=3} .
%	\]
%	By \Cref{Eq:Jacquet_and_Inclusion},
%	\[
%	\pi \hookrightarrow i_{M_j}^G \bk{i_{M_{i,j}}^{M_j}\Omega_1} .
%	\]
%	
%	Since $i_{M_{i,j}}^{M_j}\Omega_1$ is irreducible (see \Cref{Remark:Irreducibility_of_representations_of_Levi}, the case $j=1$ is similar to the case $j=6$ there).
%	Hence, $i_{M_{i,j}}^{M_j}\Omega_1\cong i_{M_{l,j}}^{M_j}\Omega_2$, where $i_{M_{l,j}}^{M_j}\Omega_2$ is the invert representation of	$i_{M_{i,j}}^{M_j}\Omega_1$.
%	Here
%	\[
%	l=\piece{8,& i=1 \\ 2, & i=3} .
%	\]
%	In the case $i=3$, we then use a similar argument once again.
%	Namely, we note that
%	\[
%	\pi \hookrightarrow i_{M_1}^G \bk{i_{M_{1,2}}^{M_1}\Omega_2} \cong i_{M_2}^G \bk{i_{M_{1,2}}^{M_2}\Omega_2}
%	\]
%	and that $i_{M_{1,2}}^{M_2}\Omega_2\cong i_{M_{2,8}}^{M_2}\Omega_3$ is an irreducible representation of $M_2$.
%	
%	
%	For the convenience of the reader, we collect the relevant data for both cases in the following tables:
%	
%	\sepline
	
%	We write a similar sequence of isomorphisms as in \Cref{Eq:Sequence_of_isomorphisms}, summarized in the following table (analog to \Cref{Table:Data_Prop:Unitary_cases}.
	
	For each case, we write a similar sequence of isomorphisms as in \Cref{Eq:Sequence_of_isomorphisms_G_level} (see also \Cref{Eq:Sequence_of_isomorphisms_unitary_case} and \Cref{Remark:Irreducibility_of_representations_of_Levi}), summarized in the following tables (analog to \Cref{Table:Data_Prop:Unitary_cases}):
	\begin{itemize}
		
		\item Fix a character $\chi$ of order $2$.
		In the case $\coset{3,-1/2,2}$:
		\begin{center}
			\footnotesize
			\begin{longtable}{|c||c|c|c|c|c|}\hline 
				$k$ & $i_k$ & $j_k$ & $r_T^{M_{i_k.j_k}}\Omega_k$ & $i_{k+1}$ & $r_T^{M_{j_{k+1},i_{k+1}}}\Omega_{k+1}$ \\ \hline
				$1$ & $1$ & $3$ & $\eeightchar{-1}{-1}{5+\chi}{-1}{-1}{-1}{-1}{-1}$ & $2$ & $\eeightchar{2+\chi}{5+\chi}{-1}{-1}{-1}{-1}{-1}{-1}$ \\ \hline
				$2$ & $2$ & $1$ & $\eeightchar{2+\chi}{5+\chi}{-1}{-1}{-1}{-1}{-1}{-1}$ & $8$ &
				$\eeightchar{-1}{5}{-1}{-1}{-1}{-1}{-1}{4+\chi}$ \\ \hline
				\caption{Data for the proof of \Cref{Prop:cosocle_len_2_E7_P2_M1_O1} in the case $\coset{3,-1/2,2}$.}
				\label{Table:Data_Prop:cosocle_len_2_P3_M1_2_O2}
			\end{longtable}
		\end{center}
		\item In the case $\coset{1,-5/2,1}$:
		\begin{center}
			\footnotesize
			\begin{longtable}{|c||c|c|c|c|c|}\hline 
				$k$ & $i_k$ & $j_k$ & $r_T^{M_{i_k.j_k}}\Omega_k$ & $i_{k+1}$ & $r_T^{M_{j_{k+1},i_{k+1}}}\Omega_{k+1}$ \\ \hline
				$1$ & $2$ & $1$ & $\eeightchar{8}{-1}{-1}{-1}{-1}{-1}{-1}{-1}$ & $8$ & $\eeightchar{-1}{5}{-1}{-1}{-1}{-1}{-1}{-2}$ \\ \hline
				\caption{Data for the proof of \Cref{Prop:cosocle_len_2_E7_P2_M1_O1} in the case $\coset{1,-5/2,1}$.}
				\label{Table:Data_Prop:cosocle_len_2_P1_M5_2_O1}
			\end{longtable}
		\end{center}
		In order to match notations with the previous case, we write $\Omega_3=\Omega_2$ for this case.
	\end{itemize}
	
	%	\begin{center}
	%		\footnotesize
	%		\begin{longtable}[H]{|c|c|c|c|c|c|c|c|}
	%			\hline
	%			$i$ & $s$ & $ord\bk{\chi}$ & $j$ & $l$ & $\lambda_0$ & $r_T^G\Omega_2$ & $r_T^G\Omega_3$ \\ \hline
	%			$1$ & $-\frac{5}{2}$ & $1$ & $2$ & $8$ & $\eeightchar{8}{-1}{-1}{-1}{-1}{-1}{-1}{-1}$ & $\eeightchar{-1}{5}{-1}{-1}{-1}{-1}{-1}{-2}$ & \\ \hline
	%			$3$ & $-\frac{1}{2}$ & $2$ & $1$ & $2$ & $\eeightchar{-1}{-1}{5+\chi}{-1}{-1}{-1}{-1}{-1}$ & $\eeightchar{2+\chi}{5+\chi}{-1}{-1}{-1}{-1}{-1}{-1}$ &
	%			$\eeightchar{-1}{5}{-1}{-1}{-1}{-1}{-1}{4+\chi}$ \\ \hline
	%		\caption{Data for the proof of \Cref{Prop:cosocle_len_2_E7_P2_M1_O1}.}
	%		\label{Table:Data_Prop:cosocle_len_2}
	%		\end{longtable}
	%%		\captionof{table}{Data for the proof of \Cref{Prop:cosocle_len_2_E7_P2_M1_O1}.}\label{Table:Data_Prop:cosocle_len_2}
	%	\end{center}
%	Here $\chi$ is a character of order $2$.
	
%	For convenience we write $\Omega_3=\Omega_2$ in the case $i=1$.
%	For convenience, we write $\Omega_3=\Omega_2$ in the case $\coset{1,-5/,2,1}$.
	The conclusion of the above discussion is that in both cases it holds that
	\[
	\pi \hookrightarrow i_{M_{2,8}}^G\Omega_3 \cong i_{M_8}^G \bk{i_{M_{2,8}}^{M_8}\Omega_3} .
	\]
	The representation $i_{M_2,8}^{M_8}\Omega_3$ is reducible and admits a maximal semi-simple subrepresentation of length $2$ (this is the case denoted by $\coset{2,-1,1}$ in \cite{SDPS_E7}).
	It follows that $i_{M_2,8}^G\Omega_3$ admits a maximal semi-simple subrepresentation of length at least $2$.
%	\todonum{Can we merge this with the following proposition?}
	
	On the other hand,
	\[
	mult\bk{\lambda_1,r_T^G\bk{i_{M_{2,8}}^G\Omega_3}} = mult\bk{\lambda_1,r_T^G\pi} = 2,
	\]
	where $\lambda_1=r_T^{M_{2,8}}\Omega_3$.
	Hence, $i_{M_{2,8}}^G\Omega_3$ admits a maximal semi-simple subrepresentation of length precisely $2$ and both of these subrepresentation intersect $\pi$, from which the claim follows.	
\end{proof}

The following is a slight variation of the previous argument.

\begin{Prop}
	\label{Prop:cosocle_len_2_D7_P5_0_1}
	The representation $\pi=i_{M_7}^G\Omega_{7,-3/2}$ admits a maximal semi-simple subrepresentation of length $2$.
%	In the case $[7,-3/2,1]$, $\pi$ admits a maximal semi-simple subrepresentation of length $2$.
\end{Prop}

\begin{proof}
	Let $\Omega_0=\Omega_{7,-3/2}$ and $\lambda_0=r_T^G\Omega_0$.
	We note that
	\[
	mult\bk{\lambda_0,r_T^G\pi} = 2 .
	\]
	Hence, $\pi$ admits a maximal semi-simple subrepresentation of length at most $2$.
	
	We write a similar sequence of isomorphisms to that of \Cref{Eq:Sequence_of_isomorphisms_G_level} (see also \Cref{Eq:Sequence_of_isomorphisms_unitary_case} and \Cref{Remark:Irreducibility_of_representations_of_Levi}), summarized in the following table (analog to \Cref{Table:Data_Prop:Unitary_cases}):
	\begin{center}
		\footnotesize
		\begin{longtable}{|c||c|c|c|c|c|}\hline 
			$k$ & $i_k$ & $j_k$ & $r_T^{M_{i_k.j_k}}\Omega_k$ & $i_{k+1}$ & $r_T^{M_{j_{k+1},i_{k+1}}}\Omega_{k+1}$ \\ \hline
			$1$ & $3$ & $7$ & $\eeightchar{-1}{-1}{-1}{-1}{-1}{-1}{7}{-1}$ & $4$ & $\eeightchar{-1}{-1}{5}{-2}{-1}{-1}{-1}{-1}$ \\ \hline
			$2$ & $4$ & $3$ & $\eeightchar{-1}{-1}{5}{-2}{-1}{-1}{-1}{-1}$ & $1$ & $\eeightchar{-4}{-1}{-1}{3}{-1}{-1}{-1}{-1}$ \\ \hline
			\caption{Data for the proof of \Cref{Prop:cosocle_len_2_D7_P5_0_1}.}
			\label{Table:Data_Prop:cosocle_len_2_P7_M3_2_O1}
		\end{longtable}
	\end{center}
	It follows that
	\[
	\pi\hookrightarrow i_{M_1}^G\bk{i_{M_{1,4}}^{M_1}\Omega_3} .
	\]
	Since $i_{M_{1,4}}^{M_1}\Omega_3$ is a reducible unitary degenerate principal series of $M_1$, it follows that $i_{M_{1,4}}^{M_1}\Omega_3=\sigma_0\oplus\sigma_1$, with $\lambda_{a.d.}\leq r_T^{M_1}\sigma_0$ and $\lambda_{a.d.}\nleq r_T^{M_1}\sigma_1$, where
	\[
	\lambda_{a.d.} = \eeightchar{-1}{0}{0}{-1}{0}{0}{-1}{0} .
	\]
	On the other hand, $\lambda_1\leq r_T^{M_1}\sigma_0, r_T^{M_1}\sigma_1$, where $\lambda_1=r_T^{M_{1,4}}\Omega_3$.
	Using \Cref{Eq:gemoetric_lemma}, one checks that
	\[
	mult\bk{\lambda_1,r_T^G i_{M_{1,4}}^G\Omega_3} =
	mult\bk{\lambda_1,r_T^G \pi} = 2 .
	\]
	On the other hand,
	\[
	mult\bk{\lambda_1,r_T^G i_{M_1}^G\sigma_0} =
	mult\bk{\lambda_1,r_T^G i_{M_1}^G\sigma_1} = 1.
	\]
	It follows that $i_{M_1}^G\sigma_0$ and $i_{M_1}^G\sigma_1$ each contains a unique irreducible subrepresentation, $\pi_0$ and $\pi_1$, both of which intersect $\pi$.
	Thus,
	\[
	\pi_0\oplus \pi_1 \hookrightarrow \pi .
	\]
	
\end{proof}

\subsection{Unresolved Cases}
\label{Subsec:Unresolved_Cases}

We finish with two cases we were unable to fully resolve, these are the cases of $\pi=i_{M_2}^G\Omega_{2,-\frac{1}{2}}$ and $\pi=i_{M_5}^G\Omega_{5,-\frac{1}{2}}$.
In both cases we show that $\pi$ admits a maximal semi-simple subrepresentation of length $1$ or $2$ and that the spherical subquotient of $\pi$ is a subrepresentation.
We were, however, unable to determine the precise length of the socle.
We do, however, outline computational methods to determine this in the future, when stronger computers are readily available.

Let $\pi=i_{M_i}^G\bk{\Omega_{M_i,s,\chi}}$ with $s<0$.
Also, let $\lambda_{a.d.}$ be an anti-dominant exponent of $\pi$ and let $\pi_0$ be an irreducible subquotient of $\pi$ such that $\lambda_{a.d.}\leq r_T^G\pi_0$.
It seems to be the case that $\pi_0$ is a subrepresentation of $\pi$ automatically.
However, we were unable to find a general proof for that.
While this usually follow from a simple branching rule calculation, this line of argument was insufficient in the following case and the claim requires a more circumvent proof.

\begin{Lem}
	\label{Lem:P2S1_2O1_spherical_is_sub}
	Let $\pi=i_{M_2}^G\Omega_{2,-\frac{1}{2}}$ and let $\pi_0$ denote the unique irreducible subquotient of $\pi$ such that $\lambda_{a.d.}\leq r_T^G\pi_0$, where
	\[
	\lambda_{a.d.} = \eeightchar{0}{0}{0}{-1}{0}{0}{0}{-1} .
	\]
	Then, $\pi_0\hookrightarrow\pi$.
\end{Lem}

\begin{proof}	
	
	Since $\lambda_{a.d.}$ is the anti-dominant exponent of $\pi$, the uniqueness of $\pi_0$ follows from \cite[Lemma A.1]{SDPS_E7}.
	In fact, it follows that $\pi_0$ appears in $i_T^G\lambda_{a.d.}$ with multiplicity $1$.
	
	We fix the following exponents of $\pi$
	\[
	\begin{array}{l}
		\lambda_0 = \eeightchar{-1}{7}{-1}{-1}{-1}{-1}{-1}{-1} \\
		\lambda_1 = \eeightchar{-1}{-1}{-1}{-1}{-1}{5}{-1}{-1} \\
		\lambda_2 = \eeightchar{-1}{-1}{-1}{-2}{1}{4}{-1}{-1},
	\end{array}
	\]
	where $\lambda_0$ is its initial exponent.
	
	By a branching rule calculation, we have
	\[
	mult\bk{\lambda_1,r_T^G \pi} =
	mult\bk{\lambda_1,r_T^G \pi_0} = 2 .
	\]
	
	On the other hand, by a central character argument (see \cite[Lemma. 3.12]{SDPS_E6} for details) it holds that $\pi_0\hookrightarrow i_T^G\lambda_1$.
	We consider the normalized intertwining operator (see \cite[Subsection 2D]{SDPS_E7} for a short account on these operators and their properties)
	\[
	i_T^G\lambda_1 \overset{N_w}{\longrightarrow} i_T^G\lambda_0,
	\]
	where
	\[
	%	w=\s{2}\s{4}\s{3}\s{1}\s{5}\s{4}\s{2}\s{3}\s{4}\s{5} .
	w=\s{5}\s{4}\s{3}\s{2}\s{4}\s{5}\s{1}\s{3}\s{4}\s{2} .
	\]
	Note that it decomposes as $N_w\bk{\lambda_1} = N_{w'}\bk{\lambda_2} \circ N_{s_5}\bk{\lambda_1}$, where
	\[
	w'=\s{4}\s{3}\s{2}\s{4}\s{5}\s{1}\s{3}\s{4}\s{2}
	\]
	and that $N_{w'}\bk{\lambda_2}$ is an isomorphism between $i_T^G\lambda_2$ and $i_T^G\lambda_0$.
	
	On the other hand, the operator
	\[
	i_T^G\lambda_1 \overset{N_{s_5}}{\longrightarrow} i_T^G\lambda_2,
	\]
	is not an isomorphism.
	However, \Cref{Eq:gemoetric_lemma} implies that
	\[
	mult\bk{\lambda_{a.d.},r_T^G\ker\bk{N_{s_5}\bk{\lambda_1}}} = 0 .
	\]
	It follows that $\pi_0$ is not contained in the kernel of $N_{s_5}\bk{\lambda_1}$ or $N_w\bk{\lambda_1}$ and hence $\pi_0$ is a subrepresentation of $i_T^G\lambda_0$.
	In particular, since $\pi_0$ appears in $i_T^G\lambda_0$ with multiplicity $1$, it is also a subrepresentation of $\pi$.

\end{proof}

\begin{Remark}
	\label{Rem:E8_P2_M1_2_O1}
	By \Cref{Eq:gemoetric_lemma}, $mult\bk{\lambda_0,r_T^G \pi} = 2$.
	Hence, the length of the socle of $\pi$ is either $1$ or $2$.
	We point out that, by \Cref{Tab::E7::COMP::P_6}, $\pi_0$ is also an irreducible subquotient of $i_{M_6}^G\Omega_{6,-1}$.
	It follows from \Cref{Eq:gemoetric_lemma} that $mult\bk{\lambda_0,r_T^G i_{M_6}^G\Omega_{6,-1}} = 1$ and hence also $mult\bk{\lambda_0,r_T^G \pi_0} = 1$, while on the other hand $mult\bk{\lambda_0,r_T^G \pi} = 2$.
	Hence, $\pi$ admits an irreducible subquotient $\pi_1\neq \pi_0$ of $\pi$ such that $\lambda_0\leq r_T^G\pi_1$.
	It remains to determine whether $\pi_1$ is a subrepresentation of $\pi$ or not.
	Namely, whether the length of the maximal semi-simple subrepresentation of $\pi$ is $1$ or $2$.

%\sepline 

	It seems that the methods used above are futile for this case.
	One could technically use the method of \cite[Proposition 4.9]{SDPS_E7} to determine whether $\pi_1$ is a subrepresentation of $\pi$, or not.	
	The method described there has two steps:
	\begin{itemize}
		\item Find a standard intertwining operator $N_w\bk{\lambda_0}$, such that the kernel of the operator is composed of the irreducible subrepresentations of $\pi$ other than $\pi_0$.
		
		\item Calculate the dimension $d$ of the space of Iwahori-fixed vectors in the kernel of $N_w\bk{\lambda_0}$.
		If $d=0$, then $\pi_0$ is the unique irreducible subrepresentation and if $d>0$ then $\pi$ admits a socle of length at least $2$.
	\end{itemize}
	 
	This is, however, beyond the capabilities of the computers available to us as the minimal Weyl word that could be used for this is
	\[
	w=\s{5}\s{4}\s{3}\s{2}\s{4}\s{5}\s{1}\s{3}\s{4}\s{2}.
	\]
	which is of length $10$.
	The issue is that the required computing time grows exponentially with the length of the Weyl word and the cardinalities of relevant Weyl groups and their coset spaces.
%	Generating the intertwining operator $N_w$ for $w$ of length $10$ .
	
	Bellow, we suggest a slightly simpler method to determine this in hope that it could be used in the near future.
	In the discussion, we give further indications on the properties of $\pi_1$.
	
	Let $\pi$, $\pi_0$, $\lambda_0$, $\lambda_1$ and $\lambda_{a.d.}$ be as in \Cref{Lem:P2S1_2O1_spherical_is_sub} and let $w$ be as above.
%	\[
%	w=\s{5}\s{4}\s{3}\s{2}\s{4}\s{5}\s{1}\s{3}\s{4}\s{2}.
%	\]

%	It holds that $\lambda_1=w\cdot\lambda_0$.
	Let $\Theta_0=\set{1,2,3,4,5}$.
	Since $w\in W_{\Theta_0}$, the action of the intertwining operator $N_w\bk{\lambda_0}$ factors through any of $M_6$, $M_7$ or $M_8$.
	We will describe a unified argument for a calculation that can be performed for each choice of $M_j$, with $j\in\set{6,7,8}$, to factor through.
	We will then explain why factorizing through $M_8$ seems to be the most efficient.
	
	Write $\Omega_0=\Omega_{2,-\frac{1}{2},1}$ and $\Omega_1=\jac{M_{2,j}}{M_2}{\Omega_0}$.
	
	We now recall, from \cite{SDPS_E6,SDPS_E7,MR2017065}, that
	\[
	length \bk{i_{M_{2,j}}^{M_j}\Omega_1} = \piece{2,& j=6,7 \\ 3,& j=8 .}
	\]
	
	More precisely:
	\begin{itemize}
		\item If $j=6$ or $j=7$, $\tau = i_{M_{2,j}}^{M_j}\Omega_1$ is of length $2$ and one may write a non-splitting exact sequence
		\[
		\tau_0 \hookrightarrow \tau \twoheadrightarrow \tau_1 ,
		\]
		where $\tau_0$ and $\tau_1$ are irreducible and $\tau_1$ is spherical.
		In the case $j=6$, the quotient $\tau_1$ is in-fact $1$-dimensional.
		
		\item If $j=8$, $\tau=i_{M_{2,8}}^{M_8}\Omega_1$ is of length $3$ and one may write a non-splitting exact sequence
		\[
		\tau_0 \hookrightarrow \tau \twoheadrightarrow \tau_1 \oplus \tau_{-1},
		\]
		where $\tau_0$, $\tau_{-1}$ and $\tau_1$ are irreducible and $\tau_1$ is spherical.
	\end{itemize}
	
	Since 
	\[
	mult\bk{\lambda_1,r_T^{M_j} \tau} =
	mult\bk{\lambda_1,r_T^{M_j} \tau_1} = 1 ,
	\]
	we conclude that the kernel and image of the intertwining operator $N_w^{M_j}\bk{\lambda_0}$ is given by:
	\[
	I=\Image\bk{N_w^{M_j}\bk{\lambda_0}}=\tau_1, \quad
	K=\ker\bk{N_w^{M_j}\bk{\lambda_0}} = \piece{\tau_0,& j=6,7 \\ \tau_0+\tau_{-1},& j=8 .}
	\]
	
	%	We thus have the following commutative diagram
	%	\[
	%	\xymatrix{
		%		\pi \ar@{^{(}->}[r] & i_{M_{2,j}}^{M_j}\Omega_1 \ar@{^{(}->}[r] \ar[d]^{N_w\bk{\lambda_0}} & i_T^G\bk{\lambda_0} \ar[d]^{N_w^{M_j}\bk{\lambda_0}} \\
		%		& i_{M_{5,j}}^{M_j}\Omega_2 \ar@{^{(}->}[r] & i_T^G\bk{\lambda_0}
		%	}
	%	\]
	%	where $\Omega_2$ satisfies $r_T^{M_{5,j}}\Omega_2=\lambda_1$.
	
	Hence, it holds that
	\[
	i_{M_j}^G\bk{K} \hookrightarrow i_{M_{2,j}}^G\bk{\Omega_1} \twoheadrightarrow i_{M_j}^G\bk{I} .
	\]
	On the other hand, since
	\[
	mult\bk{\lambda_1,r_T^G \pi_0} =
	mult\bk{\lambda_1,r_T^G i_{M_{2,j}}^G\bk{\Omega_1}} =
	mult\bk{\lambda_1,r_T^G \pi} = 2 ,
	\]
	it follows that $\pi_0$ is the unique irreducible subrepresentation of  $i_{M_j}^G\bk{I}$.
	
	It follows from the above discussion that $\pi$ admits a unique irreducible subrepresentation if and only if $\pi \cap i_{M_j}^G\bk{K} = 0$.
	
	In order to determine if the intersection is indeed $0$, we consider
	\[
	\dim_\C \coset{\pi^{\mathcal{J}} \cap i_{M_j}^G\bk{K}^{\mathcal{J}}},
	\]
	where $\mathcal{J}$ denotes the Iwahori-Hecke algebra of $G$.
	That is, we calculate the intersection of the $17,280$-dimensional space $\pi^{\mathcal{J}}$ with $i_{M_j}^G\bk{K}^{\mathcal{J}}$ whose dimension depends on $j$ and is given by
	\[
	\dim_\C \bk{i_{M_j}^G\bk{K}^{\mathcal{J}}} = \coset{\dim_\C \bk{ r_T^{M_j}\bk{K} } } \times \Card{W^{M_j,T}} .
	\]
	We collect relevant data in the following table:
	
%	The main issue to consider in choosing the value of $j$ from 6, 7 and 8, is the dimension of $i_{M_j}^G\bk{K}^{\mathcal{J}}$.
%	While
%	\[
%	\dim_\C\bk{\pi^{\mathcal{J}}} = 17,280
%	\]
%	we have
%	\[
%	\dim_\C \bk{i_{M_j}^G\bk{K}^{\mathcal{J}}} = \coset{\dim_\C \bk{ r_T^{M_j}\bk{K} } } \times \Card{W^{M_j,T}}
%	\]
%	We collect relevant data in the following table:	
%	\begin{center}
%		\begin{longtable}{|c|c|c|c|} 
%%			\nopagebreak
%			\hline 
%			$j$ & $6$ & $7$ & $8$ \\ \hline
%			$\dim_\C r_T^{M_j}\bk{K}$ & 1919 & 15 & 120=105+15 \\ \hline %K^{\mathcal{J}}
%			$\Card{W^{M_j,T}}$ & 60,480 & 6,720 & 240 \\ \hline
%			$\dim_\C \bk{i_{M_j}^G\bk{K}^{\mathcal{J}}}$ & 116,061,120 & 100,800 & 28,800  \\ \hline
%		\end{longtable}
%		\label{Table:P2_M1_2_O1_Sizes_of_cosets}
%	\end{center}
	
	\begin{center}
		\begin{tabular}{|c|c|c|c|}\hline 
			$j$ & $6$ & $7$ & $8$ \\ \hline
			$\dim_\C r_T^{M_j}\bk{K}$ & 1919 & 15 & 120=105+15 \\ \hline %K^{\mathcal{J}}
			$\Card{W^{M_j,T}}$ & 60,480 & 6,720 & 240 \\ \hline
			$\dim_\C \bk{i_{M_j}^G\bk{K}^{\mathcal{J}}}$ & 116,061,120 & 100,800 & 28,800  \\ \hline
		\end{tabular}
%		\label{Table:P2_M1_2_O1_Sizes_of_cosets}
	\end{center}
	
	Thus, while the choice of $j=6$ seems more intuitive, it is seems like using $j=8$ would be more efficient.
	
\end{Remark}

We now turn to deal with the other unresolved case, $\pi=i_{M_5}^G\Omega_{5,-\frac{1}{2}}$.
In this case, a branching rule calculation shows that the irreducible spherical subquotient $\pi_0$ of $\pi$ is a subrepresentation.
However, from \Cref{Eq:gemoetric_lemma} we have $mult\bk{\lambda_0,r_T^G \pi} = 30$, where
\[
\lambda_0 = \eeightchar{-1}{-1}{-1}{-1}{4}{-1}{-1}{-1}
\]
is the initial exponent of $\pi$.
Thus, it is not immediately clear that the length of the socle of $\pi$ is at most $2$.

\begin{Lem}
	\label{Lem:P5S1_2O1_admits_at_most_2_irr_subs}
	The representation $\pi=i_{M_5}^G\Omega_{5,-\frac{1}{2}}$ admits a maximal semi-simple subrepresentation of length at most $2$.	
\end{Lem}

\begin{proof}
	We consider the following exponents of $\pi$:
	\[
	\begin{array}{l}
		\lambda_0 = \eeightchar{-1}{-1}{-1}{-1}{4}{-1}{-1}{-1}, \\
		\lambda_1 = \eeightchar{0}{0}{0}{-1}{0}{0}{0}{1} \\
		\lambda_{a.d.} = \eeightchar{0}{0}{0}{0}{-1}{0}{0}{0} .
	\end{array}
	\]
	Here, $\lambda_0$ is the initial exponent of $\pi$, $\lambda_{a.d.}$ is its anti-dominant exponent.
	Let $\pi_0$ denote the unique irreducible subquotient of $\pi$ such that $\lambda_{a.d.}\leq r_T^G\pi_0$.
	
	Since $mult\bk{\lambda_0,r_T^G\pi} = 30$ and since a branching rule calculation yielded only $mult\bk{\lambda_0,r_T^G\pi_0}\geq 4$, the claim is not immediate.
	However, this does imply that $\pi_0\hookrightarrow\pi$.
	
	%	Using a branching rule calculation and a central character argument, one checks that $\pi_0$ is a subrepresentation of $\pi$.
	
	On the other hand, \cite[Lemma A.1]{SDPS_E7} implies that for every $\sigma\in Rep\bk{G}$, it holds that
	\[
	288 \divides mult\bk{\lambda_1,r_T^G\sigma} .
	\]
	Let $\Omega_1=r^{M_5}_{M_{5,8}} \Omega_{5,-1\frac{1}{2}}$ and note that $i_{M_{5,8}}^{M_8}\Omega_1$ is irreducible.
	%	\Cref{Eq:Jacquet_and_Inclusion} and 
	Hence, reasoning as in \Cref{Prop:Unique_Irr_Subs_Jac_and_Ind}, it follows that
	\[
	\pi \hookrightarrow i_{M_8}^G \bk{i_{M_{5,8}}^{M_8}\Omega_1}\hookrightarrow i_T^G\lambda_1
	\]
	as $\lambda_1 = r_T^{M_{5,8}}\Omega_1$.
	By \Cref{Eq:gemoetric_lemma},
	\[
	mult\bk{\lambda_1,r_T^G\pi} = 864 
	\]
	and on the other hand, a branching rule calculation yields,
	\[
	mult\bk{\lambda_1,r_T^G\pi_0} \geq 576 .
	\]
	Hence, $\pi$ could have at most one more subquotient $\pi_1$ such that $\pi_1\hookrightarrow i_T^G\lambda_1$.
	Thus, the length of the maximal semi-simple subrepresentation of $\pi$ is at most $2$.
\end{proof}

\begin{Remark}
	Attempting to follow the methods of \cite[Proposition 4.9]{SDPS_E7} in order to determine the length of the socle of $\pi$ is, too, beyond the capabilities of current available computers and even more so.
	The exponent $\lambda$ ``closest'' to $\lambda_0$ such that a branching rule calculation guarantees that $mult\bk{\lambda,r_T^G\pi_0}=mult\bk{\lambda,r_T^G\pi}$ is $\lambda=\lambda_{a.d.}$
	and the shortest Weyl element $w$ such that $w\cdot\lambda_0=\lambda$ is
	\[
	w= \s{5} \s{6} \s{7} \s{8} \s{4}\s{3}\s{2}\s{4}\s{5}\s{6}\s{7}\s{4}\s{1}\s{3} \s{2}\s{4}\s{5}\s{6}\s{4}\s{1}\s{3}\s{2}\s{4}\s{5}
%	\s{5} \s{4} \s{2} \s{3} \s{1} \s{4} \s{6} \s{5} \s{4} \s{2} \s{3} \s{1} \s{4} \s{7} \s{6} \s{5} \s{4} \s{2} \s{3} \s{4} \s{8} \s{7} \s{6} \s{5},
%	w=w\coset{5, 4, 2, 3, 1, 4, 6, 5, 4, 2, 3, 1, 4, 7, 6, 5, 4, 2, 3, 4, 8, 7, 6, 5}
	\]
	which is of length $24$ and thus, the associated intertwining operator cannot be realistically generated in currently available computers.
	Also, since this word contains all $8$ generators of $W$, the associated operator does not factor via a Levi subgroup and the method suggested in \Cref{Rem:E8_P2_M1_2_O1} is not applicable here.
	
	Here, we are able to show that the same calculation can be performed with a Weyl element of length $21$.
	While this is an improvement, this calculation still seems to be unfeasible.
%	Here, too, we suggest a more efficient method to determine whether the socle of $\pi$ is of length $1$ or $2$.
	
	As noted above,
	\[
	\pi \hookrightarrow i_T^G\lambda_1 = 
	i_{M_4}^{G} \bk{i_T^{M_4}\lambda_1}.
	\]
	We also note that, by \cite[Lemma A.4 and Equation (OR)]{SDPS_E7}, $i_T^{M_4}\lambda_1$ is non-semi-simple of length $2$, we write
	\[
	\sigma_1 \hookrightarrow i_T^{M_4}\lambda_1 \twoheadrightarrow \sigma_0 ,
	\]
	where $\sigma_0$ is spherical.
	
	Let
	\[
	\lambda_2 = \eeightchar{0}{0}{0}{-1}{0}{0}{-1}{1}.
%	\eeightchar{0}{0}{0}{-1}{1}{-1}{0}{0}
	\]
%	By \cite[Lemma A.4]{SDPS_E7} and \cite[Corollary 4.4]{SegalSingularities}, it holds that $i_T^{M_4}=\sigma_0\oplus \sigma_1$.
	We show that $i_T^{M_4}\lambda_2=\sigma_0\oplus \sigma_1$ and that $\pi_i$ is the unique irreducible subrepresentation of $i_{M_4}^G \sigma_i$.
	
	Indeed, let $\Omega_2$ denote the $1$-dimensional representation of $L_5$ such that $r_{T}^{L_5} \Omega_2 = \lambda_2$.
	By \cite[Corollary 4.4]{SegalSingularities}, both $i_{L_5}^{M_4}\Omega_2$ and $i_{L_5}^{M_4}\bk{\Omega_2\otimes St_{L_5}}$ admit a unique irreducible subrepresentation.
	Similarly, so do $i_{L_5}^G\Omega_2$ and $i_{L_5}^G\bk{\Omega_2\otimes St_{L_5}}$.
	
	On the other hand, since
	\[
	i_T^{L_5} \lambda_2 = \Omega_2 \oplus \bk{\Omega_2\otimes St_{L_5}},
	\]
	it follows that $\sigma_0=i_T^{L_5}\Omega_2$ and $\sigma_1=i_T^{L_5}\bk{\Omega_2\otimes St_{L_5}}$.
	We thus conclude that $\pi_i$ is the unique irreducible subrepresentation of $i_{M_4}^G \sigma_i$.
%	$\pi_0$ is the unique irreducible subrepresentation of $i_T^{L_5}\Omega_2$ and $\pi_1$ is that of $i_T^{L_5}\bk{\Omega_2\otimes St_{L_5}}$, where $St_{L_5}$ denotes the Steinberg representation of $L_5$.
	
%	We also conclude that
%	\[
%	i_{M_4}^G \sigma_1 \hookrightarrow i_T^G\lambda_1 \twoheadrightarrow i_{M_4}^G \sigma_0
%	\]
%	and hence $\pi_1\hookrightarrow i_T^G\lambda_1$.
%	By a central character argument, see \cite[Lemma. 3.12]{SDPS_E6}, $\pi_0\hookrightarrow i_T^G\lambda_1$.
	
%	Since $\lambda_1\leq r_T^G \pi_1$, a central character argument (see \cite[Lemma. 3.12]{SDPS_E6}) implies that $\pi_1\hookrightarrow i_T^G\lambda_1$.
	
%	\sepline
%	
%	Assume that $mult\bk{\lambda_1,r_T^G\pi_0} = 864$.
%	Since $\pi \hookrightarrow i_T^G\lambda_1$, a central character argument, see \cite[Lemma. 3.12]{SDPS_E6}, this implies that $\pi_0$ is the unique irreducible subrepresentation of $\pi$.
%	Thus, it remains to prove that $mult\bk{\lambda_1,r_T^G\pi_0} = 864$.
	
	We now note that $\sigma_1 = \ker\bk{N^{M_4}_{\s{8}\bk{\lambda_1}}}$.
	On the other hand, $\pi\hookrightarrow i_T^G\lambda_1$.
	If $\pi_1$ is a subrepresentation of $\pi$, then the kernel of $N_{w'}\bk{\lambda_0}\res{\pi}$ is non-trivial, where
	\[
	w' = \s{8} \s{4}\s{3}\s{2}\s{4}\s{5}\s{6}\s{7}\s{4}\s{1}\s{3} \s{2}\s{4}\s{5}\s{6}\s{4}\s{1}\s{3}\s{2}\s{4}\s{5} .
%	[4, 3, 2, 4, 5, 6, 7, 4, 1, 3, 2, 4, 5, 6, 4, 1, 3, 2, 4, 5]
	\]
	Here, too, all $8$ generators of $W$ appear and thus, the method suggested in \Cref{Rem:E8_P2_M1_2_O1} cannot be applied to this case.

\end{Remark}

\appendix

\section{Branching Rules Database}
\label{App:Database}

In this section, we retain the notations of \Cref{Subsec:Notations} and list data on irreducible representations of Levi subgroups which is used for the branching rule calculations performed for this paper.

That is, we wish to list triples of $\bk{\lambda,M,\tau}$ where $\lambda$ is a character of $T$, $M$ is a Levi subgroup of $G$ and $\tau$ is the unique irreducible representation of $M$ such that $\lambda\leq r_T^M \tau$.

We start, by a general rule.
For a character $\lambda$ of $T$, set
\[
\Theta_{\lambda} = \set{\alpha\in\Delta_L \mvert \gen{\lambda,\check{\alpha}}=0}.
\]
It holds that $M=M_{\Theta_{\lambda}}$ admits a unique irreducible representation $\tau$ such that $\lambda \leq r_T^M\tau$, as shown in \cite[Lemma 2.3]{SegalSingularities}, which satisfy
\[
\coset{r_T^{M_{\Theta_{\lambda}}} \tau} = \Card{W_{M_{\Theta_{\lambda}}}} \times \coset{\lambda}.
\]
This can be decoded by the \emph{Orthogonality Rule} which states that, for any irreducible representation $\sigma$ of $G$ it holds that
\[
\lambda \leq \sigma \Rightarrow \Card{W_{M_{\Theta_{\lambda}}}} \times \coset{\lambda} \leq r_T^G \sigma .
\]

For other branching rules however, instead of listing the data by Levi subgroups $M$ of $G$, it is more convenient to list them by simple factors $L$ of the derived subgroups $M^{der}$.
Moreover, we recall that the uniqueness of $\tau$ is not required, only the uniqueness of $\coset{r_T^L\tau}$ is.
Thus we list the semi-simplified Jacquet functors $\coset{r_T^L\tau}$ of irreducible representations.
Finally, it would be more convenient to write these in the ``intrinsic coordinates'' of a maximal split torus of $L$ instead of those of the torus of $G$.

That is, we list triples of data $\bk{L,\lambda,\coset{r_T^G\tau}}$, where:
\begin{itemize}
	\item $L$ is a simple, split and simply-connected $p$-adic group (whose Dynkin diagram is a sub-Dynkin diagram of $E_8$), by abuse of notations we fix a maximal split torus $T$ of $L$.
	\item $\lambda$ is a character of $T$.
	\item $\coset{r_T^L\tau}$ is the unique semi-simplified Jacquet functor of an irreducible representation $\tau$ of $L$ such that $\lambda\leq r_T^G\tau$.
\end{itemize}

We separate the list of branching rules by the type of $L$.
Most of these rules are listed in \cite[Appendix A]{SDPS_E7} while the rest can be deduced from the results of \cite{SDPS_E6,SDPS_E7,MR2017065,MR1346929,MR1134591,MR0425030}.

For convenience of applying the branching rules, we list the elements appearing in $\coset{r_T^L\tau}$ using the action of the Weyl group $W=W_L$ of $L$ on the characters of $T$.
Also, we point out the each rules can be written with several variations, either due to automorphisms of the Dynkin diagra (if $L$ is of type $A_n$, $D_n$ or $E_6$) but also due to the Aubert involution (see \cite{MR1951440}), \textbf{we list only one variation of each rule}.

\subsection{$L$ of type $A_n$}

Let $L$ be a simple group of type $A_n$.
We think of it as a simple factor in the derived group $M^{der}$ of a Levi subgroup $M$ of $G$.
We fix a labeling for the Dynkin diagram of type $A_n$ and use this labeling to formulate the branching rules arising from $L$ of type $A_n$ instead of the labeling inherited from that of the Dynkin diagram of type $E_8$ given in \Cref{Subsec:GroupData}.
This labeling is given by
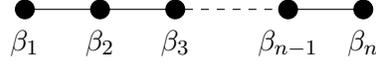
\[
\begin{tikzpicture}[scale=0.5]
	\draw (-1,0) node[anchor=east]{};
	\draw (0 cm,0) -- (4cm,0);
	\draw[dashed] (7 cm, 0) -- (4cm,0);
	\draw (7 cm,0) -- (9cm,0);
	\draw[fill=black] (0 cm, 0 cm) circle (.25cm) node[below=4pt]{$\beta_1$};
	\draw[fill=black] (2 cm, 0 cm) circle (.25cm) node[below=4pt]{$\beta_2$};
	\draw[fill=black] (4 cm, 0 cm) circle (.25cm) node[below=4pt]{$\beta_3$};
	\draw[fill=black] (7 cm, 0 cm) circle (.25cm) node[below=4pt]{$\beta_{n-1}$};
	\draw[fill=black] (9 cm, 0 cm) circle (.25cm) node[below=4pt]{$\beta_{n}$};
\end{tikzpicture}
\]

Bellow, we list the branching rules arising from $L$ of type $A_n$ which were implemented by us in the context of this paper, this list is by no means exhaustive and there are many branching rules arising from other irreducible representations of groups of type $A_n$:

\begin{itemize}
	\item If $L$ of type $A_1$, it admits a unique simple root $\beta_1$.
	For $\lambda$ such that $\gen{\lambda,\check{\beta_1}} \neq \pm1$, there exist a unique irreducible representation $\tau$ such that $\lambda\leq r_T^L\tau$.
	In fact, it holds that
	\[
	\coset{\jac{L}{T}{\tau}} = \coset{\lambda} + \coset{s_{\beta_1} \cdot \lambda}.
	\]
	This could be encoded as follows
	\begin{equation}
		\label{Eq:A1_rule}
		\lambda\leq \jac{L}{T}{\tau},\ \gen{\lambda,\check{\beta_1}} \neq \pm 1 \Longrightarrow \coset{\lambda}+\coset{s_{\beta_1}\cdot\lambda} \leq \coset{\jac{L}{T}{\tau}} .
	\end{equation}
	In what follows, branching rules will be encoded in this fashion.
	
	\item For group $L$ of type $A_n$ with $n>2$, we have the following rule:
%	\[
%	\lambda\leq \jac{L}{T}{\tau}, \, \lambda=\pm\fun{\beta_1}
%	\Longrightarrow
%	\coset{\jac{L}{T}{\tau}} = \sum_{w \in W^{L,T}} 
%	(n-l(w)) \cdot (n-1)! \coset{w\cdot \lambda} ,
%	\]
	\[
	\piece{\lambda\leq \jac{L}{T}{\tau}, \\
	\gen{\lambda,\check{\beta_1}} = \pm 1,\\ 
	\gen{\lambda,\check{\beta_k}} = 0 \quad \forall\ 2\leq k\leq n}
	\Longrightarrow
	\coset{\jac{L}{T}{\tau}} = \sum_{w \in W^{L,T}} 
	(n-l(w)) \cdot (n-1)! \coset{w\cdot \lambda} ,
	\]
	where $M= M_{\set{\beta_2 \dots \beta_{n}}}$.
%	where $\fun{\beta_1}$ is the fundamental weight associated with $\beta_1$ and $M= M_{\set{\beta_2 \dots \beta_{n}}}$.
	
	For example, in type $A_2$, this rule can be written as
	\[
	\piece{\lambda\leq \jac{L}{T}{\tau}, \\
	\gen{\lambda,\check{\beta_1}} = \pm 1,\\ 
	\gen{\lambda,\check{\beta_2}} = 0}
	\Longrightarrow 2\times\coset{\lambda}+\coset{\s{\beta_1}\cdot\lambda} \leq \coset{\jac{L}{T}{\tau}} .
	\]
	In type $A_3$, this rule can be written as
	\[
	\piece{\lambda\leq \jac{L}{T}{\tau}, \\
		\gen{\lambda,\check{\beta_1}} = 1,\\ 
		\gen{\lambda,\check{\beta_2}} =\gen{\lambda,\check{\beta_3}} = 0}
	\Longrightarrow 6\times\coset{\lambda} 
	+ 4\times \coset{\s{\beta_1}\cdot\lambda} 
	+ 2\times \coset{\s{\beta_2} \s{\beta_1}\cdot\lambda} \leq \coset{\jac{L}{T}{\tau}}
	\]

	\item For $L$ of type $A_3$ we have the following two rules:
	\[
	\piece{\lambda\leq \jac{L}{T}{\tau}, \\
		\gen{\lambda,\check{\beta_1}} = 1,\\ 
		\gen{\lambda,\check{\beta_2}} = 0, \\
		\gen{\lambda,\check{\beta_3}} = -1 }
	\Longrightarrow 2\times\coset{\lambda} + \coset{\s{\beta_1}\cdot\lambda} + \coset{\s{\beta_3}\cdot\lambda} + 2\times \coset{\s{\beta_1} \s{\beta_3}\cdot\lambda} \leq \coset{\jac{L}{T}{\tau}}
	\]
	and
	\[
	\piece{\lambda\leq \jac{L}{T}{\tau}, \\
		\gen{\lambda,\check{\beta_1}} = 1,\\ 
		\gen{\lambda,\check{\beta_2}} = 0, \\
		\gen{\lambda,\check{\beta_3}} = 1 }
	\Longrightarrow 2\times\coset{\lambda} + \coset{\s{\beta_1}\cdot\lambda} + \coset{\s{\beta_3}\cdot\lambda}\leq \coset{\jac{L}{T}{\tau}} .
	\]
	
	\item In type $A_4$, we have the following two rules:
	\[
	\begin{array}{c}
		\piece{\lambda\leq \jac{L}{T}{\tau}, \\
		\gen{\lambda,\check{\beta_2}} = 1,\\ 
		\gen{\lambda,\check{\beta_k}} = 0, \quad k=1,3,4} \\
	\Longrightarrow 
	12\times\coset{\lambda} 
	+ 8\times \coset{\s{\beta_2}\cdot\lambda} 
	+ 4\times \coset{\s{\beta_3} \s{\beta_2}\cdot\lambda} 
	+ 4\times \coset{\s{\beta_1} \s{\beta_2}\cdot\lambda} 
	+ 2\times \coset{\s{\beta_1} \s{\beta_3} \s{\beta_2} \cdot\lambda}\leq \coset{\jac{L}{T}{\tau}} .
	\end{array}
	\]
	and
	\[
	\begin{array}{c}
		\piece{\lambda\leq \jac{L}{T}{\tau}, \\
			\gen{\lambda,\check{\beta_1}} = 1,\\ 
			\gen{\lambda,\check{\beta_4}} = -1,\\ 
			\gen{\lambda,\check{\beta_k}} = 0, \quad k=2,3} \\
		\Longrightarrow 
		6\times\coset{\lambda} 
		+ 4\times \coset{\s{\beta_1}\cdot\lambda} 
		+ 4\times \coset{\s{\beta_4}\cdot\lambda} 
		+ 4\times \coset{\s{\beta_1} \s{\beta_4}\cdot\lambda} 
		+ 2\times\coset{\s{\beta_3}\s{\beta_4}\cdot\lambda}\\ 
		+ 2\times \coset{\s{\beta_2} \s{\beta_1}\cdot\lambda}
		+ 4\times \coset{\s{\beta_4} \s{\beta_2} \s{\beta_1}\cdot\lambda} 
		+ 4\times \coset{\s{\beta_1} \s{\beta_3} \s{\beta_4}\cdot\lambda} \leq \coset{\jac{L}{T}{\tau}} .
	\end{array}
	\]
	
	\item In type $A_5$ we have the following two rules:
	\[
	\begin{array}{c}
		\piece{\lambda\leq \jac{L}{T}{\tau}, \\
			\gen{\lambda,\check{\beta_3}} = 1,\\ 
			\gen{\lambda,\check{\beta_k}} = 0, \quad k=1,2,4,5} \\
		\Longrightarrow 
		36\times\coset{\lambda} 
		+ 24\times \coset{\s{\beta_3}\cdot\lambda} 
		+ 12\times \coset{\s{\beta_2}\s{\beta_3}\cdot\lambda} 
		+ 12\times \coset{\s{\beta_4}\s{\beta_3}\cdot\lambda} 
		+ 6\times \coset{\s{\beta_2} \s{\beta_4} \s{\beta_3} \cdot\lambda}\leq \coset{\jac{L}{T}{\tau}} .
	\end{array}
	\]
	and
	\[
	\begin{array}{c}
		\piece{\lambda\leq \jac{L}{T}{\tau}, \\
			\gen{\lambda,\check{\beta_1}} = 1,\\ 
			\gen{\lambda,\check{\beta_4}} = -1,\\ 
			\gen{\lambda,\check{\beta_k}} = 0, \quad k=2,3,5} \\
		\Longrightarrow 
		12\times\coset{\lambda} 
		+ 12\times \coset{\s{\beta_5} \s{\beta_4} \s{\beta_2} \s{\beta_1} \cdot\lambda}
		+ 8\times \coset{\s{\beta_1}\cdot\lambda} 
		+ 8\times \coset{\s{\beta_4}\cdot\lambda} 
		+ 8\times\coset{\s{\beta_1}\s{\beta_4}\cdot\lambda} \\
		+ 8\times \coset{\s{\beta_4} \s{\beta_2} \s{\beta_1}\cdot\lambda} 
		+ 8\times \coset{\s{\beta_1} \s{\beta_3} \s{\beta_4}\cdot\lambda}
		+ 8\times \coset{\s{\beta_5} \s{\beta_4} \s{\beta_1}\cdot\lambda}
		+ 4\times\coset{\s{\beta_3}\s{\beta_4}\cdot\lambda}\\
		+ 4\times \coset{\s{\beta_5} \s{\beta_4} \cdot\lambda}
		+ 4\times \coset{\s{\beta_5} \s{\beta_1} \cdot\lambda}
		+ 4\times \coset{\s{\beta_5} \s{\beta_1} \s{\beta_3} \s{\beta_4} \cdot\lambda}
		+ 2\times \coset{\s{\beta_5} \s{\beta_3} \s{\beta_4} \cdot\lambda}
		 \leq \coset{\jac{L}{T}{\tau}} .
	\end{array}
	\]
\end{itemize}

\subsection{$L$ of type $D_n$}

We fix the following labeling of the Dynkin diagram of type $D_n$. 
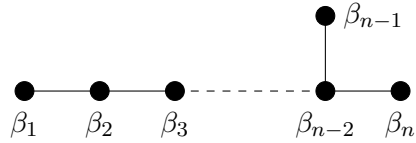
\[\begin{tikzpicture}[scale=0.5]
	%\draw (-1,0) node[anchor=east] {$ $};
	\draw (0 cm,0) -- (4 cm,0);
	\draw[dashed] (4 cm,0) -- (8 cm,0);
	\draw (8 cm,0) -- (10 cm,0 cm);
	\draw (8 cm,0) -- (8 cm,2 cm);
	\draw[fill=black] (0 cm, 0 cm) circle (.25cm) node[below=4pt]{$\beta_1$};
	\draw[fill=black] (2 cm, 0 cm) circle (.25cm) node[below=4pt]{$\beta_2$};
	\draw[fill=black] (4 cm, 0 cm) circle (.25cm) node[below=4pt]{$\beta_3$};
	\draw[fill=black] (8 cm, 0 cm) circle (.25cm) node[below=4pt]{$\beta_{n-2}$};
	\draw[fill=black] (10 cm, 0 cm) circle (.25cm) node[below=4pt]{$\beta_n$};
	\draw[fill=black] (8 cm, 2 cm) circle (.25cm) node[right=3pt]{$\beta_{n-1}$};
\end{tikzpicture}\]

If $L$ is of type $D_n$ we encode the branching rules in a similar fashion to that of type $A_n$.

\begin{itemize}
	\item For $L$ of type $D_4$, we have the following two rules:
	\[
	\begin{array}{c}
		\piece{\lambda\leq \jac{L}{T}{\tau}, \\
			\gen{\lambda,\check{\beta_2}} = 1,\\ 
			\gen{\lambda,\check{\beta_k}} = 0, \quad k=1,3,4} \\
		\Longrightarrow 
		8\times\coset{\lambda} 
		+ 5\times \coset{\s{\beta_2}\cdot\lambda} 
		+ 2\times \coset{\s{\beta_1} \s{\beta_2}\cdot\lambda} 
		+ 2\times \coset{\s{\beta_3} \s{\beta_2}\cdot\lambda} 
		+ 2\times \coset{\s{\beta_4} \s{\beta_2}\cdot\lambda} \\
		+\coset{\s{\beta_1}\s{\beta_3}\s{\beta_2}\cdot\lambda}+\coset{\s{\beta_1}\s{\beta_4}\s{\beta_2}\cdot\lambda}+\coset{\s{\beta_3}\s{\beta_4}\s{\beta_2}\cdot\lambda} 
		\leq \coset{\jac{L}{T}{\tau}} .
	\end{array}
	\]
	and
	\[
	\begin{array}{c}
		\piece{\lambda\leq \jac{L}{T}{\tau}, \\
			\gen{\lambda,\check{\beta_1}} = 1,\\ 
			\gen{\lambda,\check{\beta_4}} = -1,\\ 
			\gen{\lambda,\check{\beta_k}} = 0, \quad k=2,3} \\
		\Longrightarrow 
		12\times\coset{\lambda} 
		+ 8\times \coset{\s{\beta_1}\cdot\lambda} 
		+ 8\times \coset{\s{\beta_4}\cdot\lambda} 
		+ 12\times \coset{\s{\beta_1} \s{\beta_4}\cdot\lambda} \\
		+ 4\times\coset{\s{\beta_2}\s{\beta_4}\cdot\lambda}
		+ 4\times \coset{\s{\beta_2} \s{\beta_1}\cdot\lambda}
		\leq \coset{\jac{L}{T}{\tau}} .
	\end{array}
	\]
	
	\item For $L$ of type $D_5$, we have the following two rules:
	\[
	\begin{array}{c}
		\piece{\lambda\leq \jac{L}{T}{\tau}, \\
			\gen{\lambda,\check{\beta_5}} = 1,\\ 
			\gen{\lambda,\check{\beta_k}} = 0, \quad k=1,2,3,4} \\
		\Longrightarrow 
		120\times\coset{\lambda} 
		+ 96\times \coset{\s{\beta_5}\cdot\lambda} 
		+ 72\times \coset{\s{\beta_3} \s{\beta_5}\cdot\lambda} 
		+ 48\times \coset{\s{\beta_2} \s{\beta_3} \s{\beta_5}\cdot\lambda} \\
		+ 48\times \coset{\s{\beta_4} \s{\beta_3} \s{\beta_5}\cdot\lambda} 
		+ 32\times \coset{\s{\beta_4} \s{\beta_2} \s{\beta_3} \s{\beta_5}\cdot\lambda} 
		+ 24\times \coset{\s{\beta_1} \s{\beta_2} \s{\beta_3} \s{\beta_5}\cdot\lambda} \\
		+ 16\times \coset{\s{\beta_3} \s{\beta_2} \s{\beta_4} \s{\beta_3} \s{\beta_5}\cdot\lambda} 
		+ 16\times \coset{\s{\beta_1} \s{\beta_2} \s{\beta_4} \s{\beta_3} \s{\beta_5}\cdot\lambda} 
		+ 8\times \coset{\s{\beta_3} \s{\beta_1} \s{\beta_2} \s{\beta_4} \s{\beta_3} \s{\beta_5} \cdot\lambda}
		\leq \coset{\jac{L}{T}{\tau}} .
	\end{array}
	\]
%elif( (proj in [[0,0,0,0,-1]]) and (proj_finite==[0,0,0,0,0])):
%return { () : 120  , (i5,): 96 , (i3,i5): 72 ,(i2,i3,i5): 48,
%	(i4,i3,i5): 48 ,  (i4,i2,i3,i5): 32, (i1,i2,i3,i5): 24 ,
%	(i3,i2,i4,i3,i5): 16, (i1,i2,i4,i3,i5): 16, (i3,i1,i2,i4,i3,i5): 8} 
	and
	\[
	\begin{array}{c}
		\piece{\lambda\leq \jac{L}{T}{\tau}, \\
			\gen{\lambda,\check{\beta_3}} = 1,\\ 
			\gen{\lambda,\check{\beta_k}} = 0, \quad k=1,2,4,5} \\
		\Longrightarrow 
		24\times\coset{\lambda} 
		+ 16\times \coset{\s{\beta_3}\cdot\lambda} 
		+ 8\times \coset{\s{\beta_2} \s{\beta_3}\cdot\lambda} 
		+ 8\times \coset{\s{\beta_4} \s{\beta_3}\cdot\lambda} \\
		+ 8\times \coset{\s{\beta_5} \s{\beta_3}\cdot\lambda} 
		+ 4\times \coset{\s{\beta_4} \s{\beta_5} \s{\beta_3}\cdot\lambda} 
		+ 4\times \coset{\s{\beta_2} \s{\beta_5} \s{\beta_3}\cdot\lambda} \\
		+ 4\times \coset{\s{\beta_2} \s{\beta_4} \s{\beta_3}\cdot\lambda} 
		+ 2\times \coset{\s{\beta_4} \s{\beta_2} \s{\beta_5} \s{\beta_3}\cdot\lambda} 
		+ 2\times \coset{\s{\beta_3} \s{\beta_4} \s{\beta_2} \s{\beta_5} \s{\beta_3}\cdot\lambda}
		\leq \coset{\jac{L}{T}{\tau}} .
	\end{array}
	\]

	\item For $L$ of type $D_6$, we use one branching rule.
	For this rule, however, the list of exponents contains 30 different exponents of various multiplicities.
	Thus, we give a rough explanation on how to determine the Jacquet functor of $\tau$ in the relevant case, see \cite[Lemma A.6]{SDPS_E7} for more details.
	First, we consider the degenerate principal series representation $\sigma$ of $L$ corresponding to the notation $\coset{D_6,3,0,1}$.
	This is a direct sum of two irreducible representations, one spherical and the other is not.
	We consider the spherical irreducible constituent $\tau$, this is the unique irreducible representation of $L$ such that
	\[
	\piece{\lambda\leq r_T^L\tau \\
	\gen{\lambda,\check{\beta_1}}=\gen{\lambda,\check{\beta_4}} = -1 \\
	\gen{\lambda,\check{\beta_k}} = 0,\quad k=2,3,5,6 .} 
	\]
	Let $\tau'$ denote the non-spherical constituent of $\sigma$.
	Using the method described in \cite[Appendix B]{SDPS_E7}, one can show that
	\[
	\begin{array}{l}
		\dim_\C \bk{r_T^L\tau}=155 \\
		\dim_\C \bk{r_T^L\tau'}=5 \\
		\dim_\C \bk{r_T^L\sigma}=160 .
	\end{array}
	\]
	The initial exponent of $\sigma$ is given by
	\[
	\lambda_0 = \dsixcharchar{-1}{-1}{3}{-1}{-1}{-1} .
	\]
	Applying \Cref{Eq:A1_rule}, together with the above, yields
	\[
	\coset{r_T^L\tau'} 
	=\coset{\lambda_0} 
	+\coset{\s{\beta_3}\cdot \lambda_0}
	+\coset{\s{\beta_2}\s{\beta_3}\cdot \lambda_0}
	+\coset{\s{\beta_4}\s{\beta_3}\cdot \lambda_0}
	+\coset{\s{\beta_2}\s{\beta_4}\s{\beta_3}\cdot \lambda_0}.
	\]
	One then computes $\coset{r_T^L\sigma}$ using \Cref{Eq:gemoetric_lemma} and the structure of $\coset{r_T^L\tau}$ follows by reducing the multiplicity of these $5$ exponents by $1$.

%%%%%%%%%%%%D6
%if( (proj in [[-1,0,0,-1,0,0]]) and (proj_finite==[0,0,0,0,0,0])):
%return  { (i6,i5,i2,i3,i4,i1) : 1, (i5,i3,i4,i1):   2,
%	(i6,i5,i3,i4):   2,
%	(i1,) :  16,
%	(i4,i2,i1)  : 4,
%	(i4,i1) : 10,
%	(i4,) : 16,
%	(i5,i4,i1):  4,
%	(i4,i6,i5,i2,i3,i4,i1): 1,
%	(i3,i4,i1): 4,
%	():   24,
%	(i4,i6,i5,i3,i4):   2,
%	(i4,i6,i5,i3,i4,i1) : 1,
%	(i6,i4):   8,
%	(i3,i4,i6,i5,i2,i3,i4,i1):  1,
%	(i6,i3,i4) : 4,
%	(i6,i2,i3,i4,i1):   2,
%	(i5,i4):   8,
%	(i6,i4,i1):   4,
%	(i5,i3,i4):   4,
%	(i6,i5,i3,i4,i1):   1,
%	(i6,i5,i4):   4,
%	(i5,i2,i3,i4,i1):  2,
%	(i6,i5,i4,i1):   2,
%	(i3,i4):  8,
%	(i3,i4,i2,i1):  4,
%	(i2,i1):  8,
%	(i2,i3,i4,i1) : 4,
%	(i6,i3,i4,i1):  2,
%	(i2,i3,i4,i2,i1) : 2}

\end{itemize}

\subsection{$L$ of type $E_n$}

%If $L$ is of type $E_n$ we encode the branching rules in a similar fashion to that of type $A_n$.

In the branching rules for $L$ of type $E_n$, the list of exponents is long.
Thus, we only list the irreducible degenerate principal series of $L$ from which we derive the branching rules used by us.
One then needs to compute $\coset{r_T^L\tau}$ using \Cref{Eq:gemoetric_lemma} in order to determine the branching rule explicitly.

We fix the following labeling of the Dynkin diagram of group of type $E_6$. 

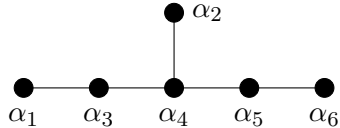
\[\begin{tikzpicture}[scale=0.5]
	\draw (-1,0) node[anchor=east]{};
	\draw (0 cm,0) -- (8 cm,0);
	\draw (4 cm, 0 cm) -- +(0,2 cm);
	\draw[fill=black] (0 cm, 0 cm) circle (.25cm) node[below=4pt]{$\alpha_1$};
	\draw[fill=black] (2 cm, 0 cm) circle (.25cm) node[below=4pt]{$\alpha_3$};
	\draw[fill=black] (4 cm, 0 cm) circle (.25cm) node[below=4pt]{$\alpha_4$};
	\draw[fill=black] (6 cm, 0 cm) circle (.25cm) node[below=4pt]{$\alpha_5$};
	\draw[fill=black] (8 cm, 0 cm) circle (.25cm) node[below=4pt]{$\alpha_6$};
	\draw[fill=black] (4 cm, 2 cm) circle (.25cm) node[right=3pt]{$\alpha_2$};
\end{tikzpicture}\]

The irreducible degenerate principal series of $L$ of type $E_6$ whose Jacquet functor $\coset{r_T^L\tau}$ contribute additional information to the one given previously are: $\coset{E_6,5,-\frac12,1}$ and $\coset{E_6,3,-\frac12,1}$.

We fix the following labeling of the Dynkin diagram of group of type $E_7$. 

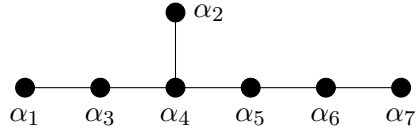
\[\begin{tikzpicture}[scale=0.5]
	\draw (-1,0) node[anchor=east]{};
	\draw (0 cm,0) -- (10 cm,0);
	\draw (4 cm, 0 cm) -- +(0,2 cm);
	\draw[fill=black] (0 cm, 0 cm) circle (.25cm) node[below=4pt]{$\alpha_1$};
	\draw[fill=black] (2 cm, 0 cm) circle (.25cm) node[below=4pt]{$\alpha_3$};
	\draw[fill=black] (4 cm, 0 cm) circle (.25cm) node[below=4pt]{$\alpha_4$};
	\draw[fill=black] (6 cm, 0 cm) circle (.25cm) node[below=4pt]{$\alpha_5$};
	\draw[fill=black] (8 cm, 0 cm) circle (.25cm) node[below=4pt]{$\alpha_6$};
	\draw[fill=black] (10 cm, 0 cm) circle (.25cm) node[below=4pt]{$\alpha_7$};
	\draw[fill=black] (4 cm, 2 cm) circle (.25cm) node[right=3pt]{$\alpha_2$};
\end{tikzpicture}\]

The irreducible degenerate principal series of $L$ whose Jacquet functor $\coset{r_T^L\tau}$ contribute additional information to the one given previously are:
$\coset{E_7,4,0,1}$, $\coset{E_7,5,0,1}$ and $\coset{E_7,5,-\frac32,1}$.

\bibliographystyle{alpha}
\bibliography{bib}

\end{document}